\documentclass[a4paper]{article}

\usepackage[english]{babel}
\usepackage[utf8x]{inputenc}
\usepackage[T1]{fontenc}

\usepackage[a4paper,top=2cm,bottom=2.5cm,left=2.3cm,right=2.3cm,marginparwidth=1.75cm]{geometry}

\usepackage[colorinlistoftodos]{todonotes}
\usepackage[colorlinks=true, allcolors=blue]{hyperref}
\usepackage{xy}
\usepackage[affil-it]{authblk}
\usepackage{amssymb}
\usepackage{cancel}
\usepackage{epstopdf}
\usepackage{verbatim}
\usepackage{amsmath,amscd}
\usepackage{graphicx}
\usepackage{subcaption}
\usepackage{amsthm}
\usepackage{tikz}
\usepackage{pgf}
\usepackage{mathrsfs}
\usetikzlibrary{arrows}
\usepackage[toc,page]{appendix}
\usepackage{caption}
\usepackage[colorinlistoftodos]{todonotes}

\newtheorem{theorem}{Theorem}
\newtheorem{lemma}{Lemma}
\newtheorem{prop}{Proposition}
\makeatletter
\newtheorem*{rep@theorem}{\rep@title}
\newcommand{\newreptheorem}[2]{%
\newenvironment{rep#1}[1]{%
 \def\rep@title{#2 \ref{##1}}%
 \begin{rep@theorem}}%
 {\end{rep@theorem}}}
\makeatother

\newreptheorem{theorem}{Theorem}

\newtheorem{cor}{Corollary}[section]

\theoremstyle{definition}
\newtheorem{mydef}{Definition}[section]
\newtheorem{algo}{Algorithm}[section]
\newtheorem{rem}{Remark}[section]
\newtheorem{exa}{Example}[section]

\title{Tropicalizing tame degree three coverings of the projective line}
\author{Paul Alexander Helminck}
\affil{University of Bremen,\\
ALTA, Institute for Algebra, Geometry, Topology and their Applications}

\begin{document}
\maketitle

\definecolor{qqqqff}{rgb}{0,0,1}
%$ latex TropGen3
%$ bibtex TropGen3
%$ latex TropGen3
%$ latex TropGen3
%\title{Tropicalizing degree three coverings of the projective line}
%\author{Paul Alexander Helminck}
%\affil[1]{University of California Berkeley}
%\affil{University of Bremen, ALTA, Institute for Algebra, Geometry, Topology and their Applications}

%\begin{document}
%\maketitle

\begin{abstract}
In this paper, we study the problem of tropicalizing tame degree three coverings of the projective line. Given any degree three covering $C\longrightarrow{\mathbb{P}^{1}}$, we give an algorithm that produces the Berkovich skeleton of $C$. In particular, this gives an algorithm for finding the Berkovich skeleton of a genus $3$ curve. The algorithm uses a continuity statement for inertia groups of semistable Galois coverings, which we prove first. After that we give a formula for the decomposition group of an irreducible component $\Gamma\subset{\mathcal{C}_{s}}$ for a semistable Galois covering $\mathcal{C}\longrightarrow{\mathcal{D}}$. We conclude the paper with a simple application of these $S_{3}$-coverings to elliptic curves, giving another proof of the familiar semistability criterion for elliptic curves using a natural degree three morphism to $\mathbb{P}^{1}$ instead of the usual degree two morphism. %covering instead of the usual degree $2$ m arising from the Weierstrass form. %These two give what we call the covering data. Together with the twisting data, we obtain the full Berkovich skeleton of the Galois closure. %To find the full skeleton, we %To find the full skeleton, we add twisting data %that is used in the algorithm. Furthermore,   %For  and a formula for the decomposition group of a component 
    %We give an algorithm that takes as input a degree three covering of the projective line and  
% Our approach is algorithmic, in the sense that we present a method that gives the  %We present an algorithm that takes a genus $3$ curve, presented either as a hyperelliptic curve or a plane quartic, and then gives the Berkovich skeleton of that curve.    %The problem will be solved in terms of covering   %The Galois closure of such a cover is defined, This will give a concrete algorithm for finding the Berkovich skeleton of any genus $3$ curve.  
\end{abstract}

\section{Introduction}

Let $K$ be a discretely valued field of characteristic zero with valuation ring $R$, maximal ideal $\mathfrak{m}$, residue field $k$, uniformizer $\pi$ and normalized valuation $v$ (with $v(\pi)=1$). We assume that $K$ is complete with respect to the aforementioned discrete valuation and that the residue field $k$ is algebraically closed. Furthermore,  %Throughout the paper,% we will assume that $K$ is complete with respect to the discrete valuation and 
we assume that the characteristic of $k$ is coprime to six to ensure that there is only tame ramification. %{\bf{<<VEREISTEN KROMME TOEVOEGEN>>}}

%That is, we will not encounter any wild ramification.
%<<{\bf{Meetkundig Galois toevoegen}}>>
Let $C$ be a smooth, projective, geometrically irreducible curve over $K$ and let $\phi:C\longrightarrow{\mathbb{P}^{1}}$ be any degree three covering. That is, the corresponding injection of function fields
\begin{equation}
K(\mathbb{P}^{1})\longrightarrow{K(C)}
\end{equation}
has degree three. We now take the Galois closure $L$ of this field extension and consider the normalization $\overline{C}$ of $C$ inside $L$. If $L=K(C)$, then $\phi$ is Galois and we see that $C$ admits an abelian degree three covering to $\mathbb{P}^{1}$ over $K$. If $\overline{C}$ is geometrically reducible, then $C$ admits an abelian degree three covering to $\mathbb{P}^{1}$ over a quadratic extension of $K$. 
%<<{\bf{controleren}}>> 
%This is geometrically irreducible i %This normalization is geometrically irreducible if and only if the normalization of $K(\mathbb{P}^{1})$ in the quadratic subfield is geometrically irreducible. One then quickly finds that thonly happens when the discriminant of $K(C)/K$ is an element of $K$ that is not a square. In t  %<<{\bf{controleren}}>>%   This has a quadratic subfield%This field extension is quite often not {\it{Galois}}, i.e. normal and separable. We take the Galois closure $L$ %$K(\overline{C})$%
 %of this field extension, corresponding to a smooth curve $\overline{C}$. If the extension was already Galois to begin with, then we just have $\overline{C}=C$. 
 In both of the above cases, the curve $C$ is a so-called {\it{superelliptic curve}}. %with a degree three Galois covering to $\mathbb{P}^{1}$. 
 The problem of tropicalizing these coverings was studied extensively in \cite{supertrop}, so we will not treat it here.%will only be interested in the case that  $\phi$ is {\it{not}} Galois.

Assuming that $\overline{C}$ is geometrically irreducible, we find that
%In the case that $\overline{C}\neq{C}$, 
%we find that
 the induced morphism $\overline{\phi}: \overline{C}\longrightarrow{\mathbb{P}^{1}}$ is geometrically Galois with Galois group $S_{3}$.  This group is solvable, so we find that the corresponding field extension $K(\mathbb{P}^{1})\subset{K(\overline{C})}$ consists of two consecutive abelian extensions. One of them induces a degree $2$ covering $D\longrightarrow\mathbb{P}^{1}$ and the other a degree $3$ {\it{abelian}} covering $\overline{C}\longrightarrow{D}$. The goal of this paper is to explain how to find the Berkovich skeleton of $\overline{C}$ using data given by the covering $\phi$. Taking the quotient under the subgroup of order $2$ corresponding to $C$ then also yields the Berkovich skeleton of $C$.\\
We will express the Berkovich skeleta of $\overline{C}$ and $C$ in terms of {\it{covering data}} and {\it{twisting data}}. The covering data consists of a conjugacy class of a decomposition group $D_{x}$ for every vertex or edge $x$ of the canonical tropical tree associated to the covering. By standard group theory, this then gives the number of edges and vertices in the pre-image of every edge and vertex. The twisting data will consist of a $2$-cocycle (in terms of graph cohomology) on the Berkovich skeleton of $D$. This twisting data then tells us how to connect the edges and vertices we obtain from the covering data.\\
The approach is similar to (and in fact inspired by) the one used in number theory for irreducible cubic polynomials. There, one is interested in the number of solutions modulo $p$ to an equation of the form $f=0$, where $f$ is a (monic) cubic polynomial in $\mathbb{Z}[x]$. If the corresponding number field $L:=\mathbb{Q}[x]/(f)$ is Galois (and thus abelian), then the number of solutions modulo $p$ is governed by a congruence relation by Kronecker-Weber and some algebraic number theory. For example, the cubic polynomial $f=x^3-3x+1$ splits into three factors modulo $p$ if and only $p\equiv{1,8}\mod{9}$, coming from the fact that $L\subset{\mathbb{Q}(\zeta_{9})}$. If $L$ is not Galois, then one considers the Galois closure. Using {\it{class field theory}}, one then finds congruence conditions on primes of the ring of integers of the quadratic subfield that solve the decomposition problem. This does {\it{not}} give simple congruence conditions on $p$ however. For instance, for $p\equiv{1}\text{ mod }{3}$ we have that $x^3-2=0$ has three solutions modulo $p$ if and only if $p=x^2+27y^2$. This last condition holds if and only if for $\mathcal{O}_{K}=\mathbb{Z}[\zeta_{3}]$ there is a generator $\pi$ of a prime $\mathfrak{p}\in\text{Spec}(\mathcal{O}_{K})$ of norm $p$ equivalent to $1\text{ mod }6\mathcal{O}_{K}$, which is the congruence condition obtained from class field theory (or simpler methods, as in \cite[Proposition 9.6.2]{Ireland1990}). %For the tropical variant, the splitting 

The paper is structured as follows. We first prove a continuity statement for inertia groups corresponding to edges and a formula for the decomposition group of a vertex. Forms of these statements already appeared in \cite{supertrop}, where they were proved directly for the case of superelliptic curves. Here, we will prove them for general Galois coverings. We then study inertia groups for tame degree three coverings of discrete valuation rings. Using the Laplacian operator, we can then express the covering data for every vertex and edge purely in terms of the branch points. After that, we consider the problem of determining the twisting data. This allows us to fully reconstruct the Berkovich skeleton of the Galois closure of the degree three covering. We conclude the paper with a proof of the familiar criterion 
\begin{equation}
"v(j)<0 \text{ if and only if }E\text{ has multiplicative reduction over an extension of }K"
\end{equation}
using the natural $3:1$ covering given by $(x,y)\mapsto{y}$ on an elliptic curve in Weierstrass form $y^2=x^3+Ax+B$. Here, $j$ is the $j$-invariant of the elliptic curve $E$. See \cite[Chapter VII, Proposition 5.5]{Silv1} for a proof using the hyperelliptic covering.  

There are three appendices, containing some results on normalizations and separating semistable models that the author couldn't find a reference for.

%They are included for the convenience of the reader. %For a lack of reference, they are included to keep the paper self-contained. 

\section{Preliminaries}

Let $K$ be a discretely valued field of characteristic zero with valuation ring $R$, uniformizer $\pi$, maximal ideal $\mathfrak{m}=(\pi)$, residue field $k$ and normalized valuation $v$ with $v(\pi)=1$. We assume that $K$ is complete with respect to this valuation and that $k$ is algebraically closed.  We also assume that the characteristic of the residue field is coprime to six. For any finite extension $K'$ of $K$, we let $R'$ be a discrete valuation ring in $K'$ dominating $R$.

Let $C$ be a smooth, projective, geometrically irreducible curve over $K$ with a degree three covering  $\phi:C\longrightarrow{\mathbb{P}^{1}}$. This means that the injection of function fields %be any degree three covering. That is, the corresponding injection of function fields
\begin{equation}
K(\mathbb{P}^{1})\longrightarrow{K(C)}
\end{equation}
has degree three. We will write $K(\mathbb{P}^{1})=K(x)$ and $K(C)=L$ from now on. We can then find an element $z\in{L\backslash{K(x)}}$ %{K(C)\backslash{K(\mathbb{P}^{1})}$ 
that satisfies %, its minimal polynomialwe obtain an equation that can easily be transformed to 
\begin{equation}
z^3+pz+q=0
\end{equation}
for $p,q\in{K(x)}$. We now assume that the corresponding {\it{Galois closure}} $\overline{L}$ has Galois group $S_{3}$ over $K(x)$. %This might still lead to some complications however. 
Let $\overline{C}$ be the normalization of $C$ in $\overline{L}$. We then have
\begin{lemma}\label{GeomIrr}
$\overline{C}$ is geometrically irreducible if and only $\Delta=4p^3+27q^2\notin{K}$.
\end{lemma} %Let $\overline{L}$ be   %This is equivalent to the discriminant
\begin{proof}
Note that being geometrically irreducible is equivalent to $\overline{L}\cap{\overline{K}}=K$ by \cite[Chapter 3, Corollary 2.14]{liu2}, where $\overline{K}$ is the algebraic closure of $K$. Note that $\overline{L}$ naturally contains the {\it{field}} (by assumption on the Galois group) $K(x)[y]/(y^2-\Delta)$. Suppose that $\overline{C}$ is geometrically reducible. Then there exists a $z_{0}\in\overline{L}$ such that $K(z_{0})\supset{K}$ is finite of degree $\neq{1}$. If $K(z_{0})/K$ is of degree $2$, we  reason as follows. For some $a,b\in{K}$, we have $a+bz_{0}=y$ (since there is only one subfield of degree two) and thus $y^2\in\overline{K}$. We then find $y^2\in{K(x)\cap{\overline{K}}}=K$, a contradiction. Now suppose that $K(z_{0})/K$ is of degree $3$. Then some conjugate $\sigma(z_{0})$ of $z_{0}$ belongs to $L$. But then $\sigma(z_{0})\in{L\cap{\overline{K}}}=K$, a contradiction. 

For the other direction, suppose that $\Delta\in{K}$. Then $y$ is an element of $(\overline{L}\cap{\overline{K}})\backslash{K}$. %, but not in $K$, because otherwise $\overline{L}=L$. 
This contradicts our assumption on $\overline{L}$, finishing the proof.  
\end{proof}
For the remainder of the paper, we assume that $\overline{C}$ is geometrically irreducible, which is quite an easy condition to check by Lemma \ref{GeomIrr}. 
The Galois closure $\overline{L}$ can now be described by the two equations
\begin{equation}
w^3=y-\sqrt{27}q
\end{equation}
and
\begin{equation}
y^2=\Delta.
\end{equation}
See Appendix \ref{Appendix1} for the details. These equations first arose in the famous Cardano formulas, where they are used to express $z$ in terms of the above radicals. We will not use these formulas in this paper, since the above equations are enough to derive all the information we need.

\subsection{Disjointly branched morphisms}\label{DisBranSect}
Consider any finite Galois covering $C\longrightarrow{D}$ over $K$. Here, we assume that $C$ and $D$ are smooth, projective and geometrically irreducible over $K$. 
%Let $\phi:C\rightarrow{D}$ be a finite morphism of smooth, projective, geometrically connected curves over $K$. We say $\phi$ is \emph{Galois} if the corresponding morphism on function fields $K(D)\rightarrow{K(C)}$ is Galois. That is, it is normal and separable.
 By a model $\mathcal{D}$ for a curve $D$, we mean an integral normal projective scheme $\mathcal{D}$ of dimension two with a flat morphism $\mathcal{D}\rightarrow{\text{Spec}(R)}$, and an isomorphism $\mathcal{D}_{\eta}\rightarrow{D}$ of the generic fibers. Let $\mathcal{C}$ be a model for $C$ and $\mathcal{D}$ a model for $D$. A \emph{finite morphism of models for} $\phi$ is a finite morphism $\mathcal{C}\rightarrow{\mathcal{D}}$ over $\text{Spec}(R)$ such that the base change to $\text{Spec}(K)$ gives $\phi:C\rightarrow{D}$.
\begin{mydef}\label{disbran}
Let $\phi:C\rightarrow{D}$ be a finite, Galois morphism of curves over $K$ with Galois group $G$.
Let $\phi_{\mathcal{C}}:\mathcal{C}\rightarrow{\mathcal{D}}$ be a finite morphism of models for $\phi$. We say $\phi_{\mathcal{C}}$ is \emph{disjointly branched} if the following hold:
\begin{enumerate}
\item The closure of the branch locus in $\mathcal{D}$ consists of disjoint, smooth sections over $\text{Spec}(R)$.
\item Let $y$ be a generic point of an irreducible component in the special fiber of $\mathcal{C}$. Then the induced morphism $\mathcal{O}_{\mathcal{D},\phi(y)}\rightarrow{\mathcal{O}_{\mathcal{C},y}}$ is \'{e}tale. % for every generic point of an irreducible component in the special fiber of $\mathcal{C}$ $y$. 
\item $\mathcal{D}$ is strongly semistable, meaning that $\mathcal{D}$ is semistable and that the irreducible components in the special fiber are all smooth. 
\end{enumerate}
A theorem by Liu and Lorenzini \cite[Theorem 2.3]{liu_lorenzini_1999} says that if $\phi_\mathcal{C}$ is disjointly branched then $\mathcal{C}$ is actually also \emph{semistable} and \cite[Proposition 3.1]{tropabelian} shows $\mathcal{C}$ is also strongly semistable.
\end{mydef}

Let us recall the notions of \emph{inertia groups} and \emph{decomposition groups}. %We now define \emph{inertia groups} and \emph{decomposition groups} for 
Let $G$ be a finite group acting on a scheme $X$. 
For any point $x$ of $X$, we define the decomposition group $D_{x,X}$ to be $\{\sigma \in G:\sigma(x)=x\}$, the stabilizer of $x$. 
Every element $\sigma\in{D_{x,X}}$ naturally acts on $\mathcal{O}_{X,x}$ and the residue field $k(x)$. We define the \emph{inertia group} $I_{x,X}$ of $x$ to be the elements of $D_{x}$ reducing to the identity on $k(x)$. In other words, $\sigma \in I_{x,X}$ if and only if for every $z \in {O_{X,x}}$, we have $\sigma z \equiv z \mod m_x$, where $m_{x}$ is the unique maximal ideal of $O_{X,x}$.
We will quite often omit the scheme $X$ in $I_{x,X}$ and $D_{x,X}$ and just write $I_{x}$ and $D_{x}$. 

For any disjointly branched morphism $\phi_{\mathcal{C}}$, we have an induced morphism of intersection graphs 
\begin{equation}
\Sigma(\mathcal{C})\rightarrow{\Sigma(\mathcal{D})}
\end{equation}
such that $\Sigma(\mathcal{C})/G=\Sigma(\mathcal{D})$ for the natural Galois action of $G$ on $\Sigma(\mathcal{C})$, see \cite[Theorem 3.1]{tropabelian}. For $v$ a vertex in $\Sigma(\mathcal{C})$ with corresponding generic point $y$, we write $D_{v}$ and $I_{v}$ for $D_{y}$ and $I_{y}$ respectively. We sometimes also use the notation $D_{\Gamma}$ and $I_{\Gamma}$, where $\Gamma=\overline{\{y\}}$. Similarly for $e$ an edge, we write $D_{e}$ and $I_{e}$ for $D_{x}$ and $I_{x}$ respectively, where $x$ is the intersection point corresponding to $e$.%For $v$ a vertex or $e$ an edge in $\Sigma(\mathcal{C})$, we will write $D_{v}$, $I_{v}$, $D_{e}$ and $I_{e}$ for the decomposition and inertia groups corresponding to the generic point of $v$ and the closed point corresponding to $e$.  %for the 

For every irreducible component $\Gamma'\subset{\mathcal{C}_{s}}$ with image $\Gamma\subset{\mathcal{D}_{s}}$, we have a natural Galois morphism $\phi_{\Gamma'}:\Gamma'\rightarrow{\Gamma}$, where the Galois group is the decomposition group $D_{\Gamma'}$. This can be found in \cite[Section 3.2.1]{tropabelian}.  %This means that we have an injection of function
For any irreducible component $\Gamma\subset{\mathcal{D}_{s}}$, we write $k(\Gamma)$ for the function field of $\Gamma$ and similarly for $\Gamma'$. We then have an induced injection of function fields $k(\Gamma)\rightarrow{k(\Gamma')}$ that is Galois.

Returning to our $S_{3}$-covering $\overline{\phi}:\overline{C}\rightarrow{\mathbb{P}^{1}}$, we now need a semistable model of $\mathbb{P}^{1}$ that gives a disjointly branched morphism for $\overline{\phi}$. This is obtained as follows: we take a model $\mathcal{D}_{S}$ for $\mathbb{P}^{1}$ that separates the closure of the set $S=\text{Supp}(p,q,\Delta)$ in the special fiber. We will see in Lemma \ref{BranchLocus1} that this set contains the branch locus. In Appendix \ref{Appendix2}, we quickly review the construction of $\mathcal{D}_{S}$. %The model $\mathcal{D}_{S}$ then see Appendix \ref{Appendix2} for the details. 
The corresponding intersection graph of the model $\mathcal{D}_{S}$ %described in Appendix \ref{Appendix2} 
is also known as the {\it{tropical separating tree}} for $S$. See \cite[Algorithm 4.2]{supertrop} and \cite[Section 4.3]{tropicalbook} for the tropical point of view on this. 

%How to obtain this model is explained in Appendix \ref{Appendix2}. 

We take the normalization $\overline{\mathcal{C}}_{S}$ of $\mathcal{D}_{S}$ in $K(\overline{C})$ and note that the corresponding morphism $\overline{\mathcal{C}}_{S}\rightarrow{\mathcal{D}_{S}}$ can be ramified at the vertical divisors. One then takes the base change of this model $\mathcal{D}_{S}$ to $\text{Spec}(R')$, where $K'=\text{Quot}(R')$ is the tamely ramified extension of order equal to the least common multiple of the ramification indices of the vertical components in $\mathcal{D}_{S}$. % that eliminates the vertical ramification in $\overline{\mathcal{C}}_{S}\rightarrow{\mathcal{D}_{S}}$. 
The normalization $\overline{\mathcal{C}}'_{S}$ of $\mathcal{D}_{S}\times{\text{Spec}(R')}$ is then semistable and $\overline{\mathcal{C}}'_{S}\rightarrow{\mathcal{D}_{S}\times{\text{Spec}(R')}}$ is disjointly branched. See \cite[Chapter 10, Proposition 4.30]{liu2} or \cite[Theorem 2.3]{liu_lorenzini_1999}.% for this.   %that the vertical divisors in $\mathcal{D}_{S}$ are not ramified. %  finite extension of $K$.
%This set $S$ contains the branch locus of $\overline{\phi}$, so $\mathcal{D}_{S}$ and its normalization $\mathcal{C}_{S}$ satisfy the conditions for a disjointly branched model. The construction of this model $\mathcal{D}_{S}$ is explained in Appendix \ref{Appendix2}. 

Throughout the paper, we will be making use of the $\Gamma$-modified form of an element $f\in\mathcal{C}$. Assuming that $\mathcal{C}_{s}$ is reduced (which is the case for semistable $\mathcal{C}$ for instance), we find that $v_{\Gamma}(\pi)=1$ for any irreducible component $\Gamma\subset{\mathcal{C}_{s}}$. For $k=v_{\Gamma}(f)$, we than have that %Setting
\begin{equation}
f^{\Gamma}:=\dfrac{f}{\pi^{k}}
\end{equation}
has $v_{\Gamma}(f^{\Gamma})=0$. This means that we can consider the reduction of $f^{\Gamma}$ in the residue field of $\Gamma$. The divisor of this reduced form is then in fact completely determined by the horizontal and vertical divisor of $f$, see \cite[Proposition 5.1]{tropabelian}. 

   %$\Sigma(\mathcal{ %$\mathcal{D}_{\mathbb{P}^

 %The Galois morphism $\phi_{\Gamma'}$ then naturally induces a  

\subsection{The Laplacian operator for finite graphs}\label{Laplacian1}
In this section, we introduce the Laplacian operator for a finite graph $\Sigma$. We will mostly follow \cite[Section 1.3]{baker}. %and \cite{bakerfaber}. 
This Laplacian operator will be used to obtain information about the splitting behavior of edges in Sections \ref{InertTechnique} and \ref{QuadraticSubfieldTechnique}. 

Let $\Sigma$ be a graph, which we will assume to be finite, connected and without loop edges. Let $V(\Sigma)$ be its vertices and $E(\Sigma)$ its edges. We define $\text{Div}(\Sigma)$ to be the free abelian group on the vertices $V(\Sigma)$ of $G$.  %Elements in $\text{Div}(\Sigma)$ will be referred to as divisors on $\Sigma$. 
Writing $D\in{\text{Div}(\Sigma)}$ as $D=\sum_{v\in{V(\Sigma)}}c_{v}(v)$, we define the degree map as $\text{deg}(D)=\sum_{v\in{V(\Sigma)}}c_{v}$. We then let $\text{Div}^{0}(\Sigma)$ be the group of divisors of degree zero. % on $\Sigma$.

 Now let $\mathcal{M}(\Sigma)$ be the group of $\mathbb{Z}$-valued functions on $V(\Sigma)$. Define the {\it{Laplacian operator}} $\Delta:\mathcal{M}(\Sigma)\longrightarrow\text{Div}(\Sigma)$ by 
\begin{equation*}
\Delta(\phi)=\sum_{v\in{V(G)}}\sum_{e=vw\in{E(G)}}(\phi(v)-\phi(w))(v).
\end{equation*}   
By \cite[Corollary 1]{bakerfaber}, we find that $\Delta(\phi)\in\text{Div}^{0}(\Sigma)$. %The reader can think of this as being an analogue of the statement in analysis on the usual Laplacian operator that says
%\begin{equation}
%\int_{X}\Delta(f)=0.
%\end{equation}
We then define the group of principal divisors to be the image of the Laplacian operator:
\begin{equation*}
\text{Prin}(\Sigma):=\Delta(\mathcal{M}(\Sigma)).
\end{equation*}

%\begin{rem}
%We took the discrete definition of Laplacians from \cite[Section 1.4.]{baker}. Quite often we will also tacitly use the $\mathbb{Q}$-metric graph version by taking ramified extensions and considering divisors on subdivisions of our original graph.
%\end{rem}

We now make the transition from divisors on algebraic curves to divisors on graphs. Let $C$ be a smooth, projective, geometrically irreducible curve over $K$ with a strongly semistable {\it{regular}} model $\mathcal{C}$. If we have a strongly semistable model, we can always desingularize it to obtain such a model. Now take a point $P\in{C(K)}$. Then $P$ reduces to a unique component $c(P)\subset{\mathcal{C}_{s}}$ by \cite[Chapter 9, Corollary 1.32]{liu2}. This naturally extends to a map 
\begin{equation}
\rho: \text{Div}(C)\rightarrow{\text{Div}(\Sigma(\mathcal{C}))}.
\end{equation} % the closure of $P$ contains a unique point in the special fiber that is actually smooth.
We then have
\begin{lemma}
\begin{equation}
\rho(\text{Prin}(C))\subseteq{\text{Prin}(\Sigma(\mathcal{C}))}.
\end{equation}
\end{lemma}
\begin{proof}
This is \cite[Lemma 2.1]{baker}.
\end{proof}

The Lemma implies that for every $f\in{K(C)}$, we can find a $\phi$ such that $\Delta(\phi)=\rho(\text{div}(f))$. For trees, it is not too hard to find this function $\phi$, see \cite[Lemma 2.3]{supertrop}. We will sometimes refer to the function $\phi$ as the {\it{"Laplacian"}} of $f$.

\section{Inertia and decomposition groups for disjointly branched morphisms}\label{InertSection}
In this section, we prove a continuity result for inertia groups of a disjointly branched morphism, as defined in Section \ref{DisBranSect}. More precisely, for a regular subdivision $\mathcal{D}_{0}$ of $\mathcal{D}$, we will give a formula for the inertia groups of the new components in $\mathcal{D}_{0}$ in terms of the inertia group of the corresponding edge in $\mathcal{D}$. This will allow us to determine the inertia group of an edge in terms of codimension one phenomena, namely the inertia groups of the generic points of these new components. 
We will also give a formula for the decomposition group of a vertex $v'\in\Sigma(\mathcal{C})$ lying above a vertex $v\in\Sigma(\mathcal{D})$, where the corresponding component $\Gamma_{v}$ has genus zero. These formulas will be used in the algorithm to give the covering data for the morphism of graphs $\Sigma(\overline{\mathcal{C}})\rightarrow\Sigma(\mathcal{D})$. % That is, the component corresponding to $v$ has genus zero  %This formula has the additional condition %This formula has an additional requirement on $v$, na% requires the ad

\subsection{Subdivisions and inertia groups for edges}

%namely inertia groups of components in the special fiber.

Consider a disjointly branched Galois morphism $\phi:\mathcal{C}\rightarrow{\mathcal{D}}$ with $x\in\mathcal{C}$ an intersection point with length $n_{x}$ and $y$ its image in $\mathcal{D}$ with length $n_{y}$. We will denote the Galois group by $G$. From \cite[Chapter 10, Proposition 3.48]{liu2}, %(and the fact that the inertia groups here inject into $\text{Aut}_{\mathcal{O}_{K}}(\mathcal{O}_{X,x})$),
 we then have the formula
\begin{equation}\label{InertiaFormula}
n_{y}=|I_{x/y}|\cdot{n_{x}}.
\end{equation}

Let $y$ be an intersection point in $\mathcal{D}$, with corresponding components $\Gamma_{0}$ and $\Gamma_{n}$. Here $n$ is the length of $y$.
We now take a regular subdivision $\mathcal{D}_{0}$ of $\mathcal{D}$ in $y$. That is, we have a model $\mathcal{D}_{0}$ with a morphism $\psi: \mathcal{D}_{0}\rightarrow{\mathcal{D}}$ that is an isomorphism outside $y$ and the pre-image $\psi^{-1}\{y\}$ of $y$ consists of $n-1$ projective lines $\Gamma_{i}$. Here, the projective lines are labeled such that $\Gamma_{i}$ intersects $\Gamma_{i+1}$ in one point: $y_{i,i+1}$. Furthermore, we have that $\Gamma_{1}$ intersects an isomorphic copy of the original component $\Gamma_{0}$ in $y_{0,1}$ and likewise $\Gamma_{n-1}$ intersects an isomorphic copy of the original component $\Gamma_{n}$ in $y_{n-1,n}$. See \cite[Chapter 8, Example 3.53. and Chapter 9, Lemma 3.21]{liu2} %and \cite[]{liu2} 
for the details. %<<Samenvoegen Referenties>> %details of this process. 

 We now take the normalization $\mathcal{C}_{0}$ of $\mathcal{D}_{0}$ in $K(\mathcal{C})$. By virtue of the universal property for normalizations, we have a natural morphism
\begin{equation}
\mathcal{C}_{0}\rightarrow{\mathcal{C}}
\end{equation}
that is an isomorphism outside $\phi^{-1}(y)$.

Taking the tamely ramified extension $K\subset{K'}$ of order $\text{lcm }(|I_{\Gamma_{i}}|)$, we obtain a new model $\mathcal{D}'_{0}=\mathcal{D}_{0}\times_{\text{Spec}(R)}{\text{Spec}(R')}$ over $R'$, which is the normalization of $\mathcal{D}_{0}$ in $K'(\mathcal{D})$.
 Taking the normalization $\mathcal{C}'_{0}$ of this model inside $K'(\mathcal{C})$, we then naturally obtain morphisms
\begin{equation}
\mathcal{C}'_{0}\rightarrow{\mathcal{C}_{0}}\rightarrow{\mathcal{C}}.
\end{equation}
Here the first morphism is finite and the second one is birational. Note that by \cite[Chapter 10, Proposition 4.30]{liu2}, we have that $\mathcal{C}'_{0}$ is again semistable and that $G$ naturally acts on $\mathcal{C}_{0}$ and  $\mathcal{C}'_{0}$ such that $\mathcal{C}_{0}/G=\mathcal{D}_{0}$ and $\mathcal{C}'_{0}/G=\mathcal{D}'_{0}$ (which follows from the fact that $G$ acts naturally on any normalization, see \cite[Proposition 3.5]{tropabelian}). 

We now wish to study the inertia groups of the various points in $\mathcal{C}'_{0}$, $\mathcal{C}_{0}$ and $\mathcal{C}$. To do that, we will introduce the notion of a "\emph{chain}".

%%%%%%%%%%%BEGIN OUD
\begin{comment}
%Decompositie van morfismen: $\mathcal{C}\leftarrow{\mathcal{C}_{0}}\leftarrow{\mathcal{C}'_{0}}$, waarbij de eerste birationeel is en de tweede eindig. Daarna het concept van een "keten" introduceren en laten zien dat er boven elke $e\in\mathcal{C}$ maar \'{e}\'{e}n keten $\{e'_{i}\}$ hangt in $\mathcal{C}'_{0}$.

\end{comment}
%%%%%%%%EINDE OUD

\begin{mydef}
Let $y_{i,i+1}$ and $y'_{i,i+1}$ be the intersection points in $\mathcal{D}_{0}$ and $\mathcal{D}'_{0}$ respectively that map to $y\in\mathcal{D}$ under the natural morphism. Similarly, let $y_{i}$ and $y'_{i}$ be the generic points of the components in $\mathcal{D}_{0}$ and $\mathcal{D}'_{0}$ that map to $y$. Here the generic points are labeled such that $y_{i,i+1}$ is a specialization of both $y_{i}$ and $y_{i+1}$. %$\overline{\{y_{i}\}}\cap{\overline{\{y_{i+1}\}}=y_{i,i+1}$ and likewise for $y'_{i}$. 
A {\bf{chain}} lying above these points is a collection of generic points $x_{i}$ in the special fiber of $\mathcal{C}_{0}$ or $\mathcal{C}'$ and closed points $x_{i,i+1}$ in $\mathcal{C}_{0}$ or $\mathcal{C}'_{0}$ such that:
\begin{enumerate}
\item $x_{i,i+1}$ is a specialization of both $x_{i}$ and $x_{i+1}$,
\item The $x_{i,i+1}$ map to $y_{i,i+1}$,
\item The $x_{i}$ map to $y_{i}$.
\end{enumerate}
For the remainder of this section, we will refer to these simply as a "{\it{chain}}".
%Chain lying above an edge $e$ in the decomposition $\mathcal{C}\leftarrow{\mathcal{C}_{0}}\leftarrow{\mathcal{C}'_{0}}$.
\end{mydef}  

\begin{lemma}\label{LiftChains}
Let $\{x_{i,i+1}\}\cup{\{x_{i}\}}$ be a chain in $\mathcal{C}_{0}$. Then there exists a chain $\{x'_{i,i+1}\}\cup{\{x'_{i}\}}$ in $\mathcal{C}'_{0}$ mapping to $\{x_{i,i+1}\}\cup{\{x_{i}\}}$.
\end{lemma}
\begin{proof}
The idea of the proof is to apply the going-up and going-down theorems for integral extensions several times as follows.  Since $\mathcal{C}'_{0}\rightarrow{\mathcal{C}_{0}}$ is finite, the base change to the special fiber of $\mathcal{C}_{0}$ is also finite. This ensures that any lifts we obtain will be either closed points or generic points of components. We start with $x_{0}$ and $x_{0,1}$ and pick lifts $x'_{0}$ and $x'_{0,1}$ (which exist by the going-up theorem). %There exists only one other generic point $x'_{1}$ that specializes to $x'_{0,1}$ (we use the semistable structure here) and we see (??) that $x'_{1}$ maps down to $x_{1}$.  
Using the going-down theorem for $x_{0,1}$, $x_{1}$ and $x'_{0,1}$, we obtain a point $x'_{1}$ lying above $x_{1}$. %We note that $x'_{0}$ and $x'_{1}$ are generic points  
Continuing in this fashion yields the lemma.  %Since $x'_{0,1}$ is now a specialization of both $   %Since $x'_{0,1}$ must be an intersection point, we find that there exists 
\end{proof}

\begin{lemma}\label{ChainLemma}
For every intersection point $x\in\mathcal{C}$, there is only one chain in $\mathcal{C}_{0}$ and in $\mathcal{C}'_{0}$ lying above it.
\end{lemma}
\begin{proof}
Let us prove this for $\mathcal{C}'_{0}$ first. Since $\mathcal{C}'_{0}$ is semistable, we know that the morphism $\mathcal{C}'_{0}\rightarrow{\mathcal{C}}$ is just a blow-up over $R'$ in the sense that the edge $x$ is subdivided into a chain of projective lines. This gives a one-to-one correspondence between chains in $\mathcal{C}'_{0}$ and edges $x\in\mathcal{C}$ lying above $y$. This then also gives the result for $\mathcal{C}_{0}$ as follows. Since every chain in $\mathcal{C}_{0}$ is liftable to a chain in $\mathcal{C}$ (by Lemma \ref{LiftChains}), it has to be unique. Indeed, if there exist two different chains in $\mathcal{C}_{0}$ mapping to $x$, then there would be two different chains in $\mathcal{C}'_{0}$ mapping to $x$, a contradiction. %$\mathcal{C}'_{0}\rightarrow{\mathcal{C}}$ is finite.  %since every chain there is the image of only chain in $\mathcal{C}'_{0}$.   %these chains correspond to projective lines intersecting each other. 
\end{proof}

%%%%%%%%%%%%%%%%%%%%%BEGIN OUD
\begin{comment}
\begin{center}
{\bf{Het lemma hierna is volgens mij niet meer nodig. Even controleren.
}}
\end{center}
\begin{lemma}
Let $\phi:\mathcal{C}\rightarrow{\mathcal{D}}$ be a disjointly branched Galois morphism. Then
\begin{equation}
l(e')\cdot{|I_{e'/e}|}=l(e).
\end{equation}
\end{lemma}

\begin{cor}
A disjointly branched morphism $\phi$ is ramified at an edge $e'$ if and only if the length of the image $l(e)$ is greater than the length of the original $l(e')$. 
\end{cor}
\end{comment}
%%%%%%%%%%%%%%%%%%%% EINDE OUD

%%%%%% BEGIN OUD
\begin{comment}
\begin{center}
{\bf{Het volgende is verwerkt:}}
\end{center}
{\it{Criterium: een zijde vertakt bij een (Galois) morfisme van semistabiele modellen dan en slechts dan de lengte van de zijde groter wordt.}}
\end{comment}
%%%%%%%EINDE  OUD

\begin{lemma}
Let $x\in\mathcal{C}$ be an intersection point lying over $y$ %$e'$ be an edge in $\Sigma(\mathcal{C})$
 and let $x'_{i,i+1}$ and $x_{i,i+1}$ be closed points in $\mathcal{C}'_{0}$ and $\mathcal{C}_{0}$ respectively that map to $x$. %the closed point corresponding to $e'$. 
Then
\begin{equation} I_{x}=I_{x'_{i,i+1}}=I_{x_{i,i+1}}.
\end{equation} 
\end{lemma}
\begin{proof}
We will prove that $D_{e'}=D_{x'_{i,i+1}}=D_{x_{i,i+1}}$. Since the residue field $k$ is algebraically closed by assumption, we have that they are equal to their inertia groups. For any chain $\{x_{i,i+1}\}\cup{\{x_{i}\}}$ (or $\{x'_{i,i+1}\}\cup{\{x'_{i}\}}$ for $\mathcal{C}'_{0}$) mapping to $x$ and $\sigma\in{G}$, we have the induced chain $\{\sigma(x_{i,i+1})\}\cup{\{\sigma(x_{i})\}}$, which maps down to $\sigma(x)$. Using this and Lemma \ref{ChainLemma}, we immediately obtain the desired result. 

%This follows from Lemma \ref{ChainLemma}, compatibility of the Galois action with the morphisms $\mathcal{C}'_{0}\rightarrow{\mathcal{C}_{0}}\rightarrow{\mathcal{C}}$ and the fact that the inertia groups are the same as their decomposition groups by assumption on the residue field.  % <<{\bf{Details aanvullen}}>>

%Er kan nu verwezen worden naar Lemma \ref{ChainLemma}. 

%Het idee is als volgt. Er is maar \'{e}\'{e}n keten in $\mathcal{C}'_{0}$ die afbeeldt op $e$, vanwege standaardstellingen m.b.t. semistabiliteit. Dit geldt dan ook voor elke $e$. De actie van $G$ is compatibel met deze identificatie van ketens en zijden.

%Stel nu dat $\sigma(e)=e$, terwijl $\sigma(e'_{i})\neq{e'_{i}}$. Dan wordt de keten $\{e'_{i}\}$ afgebeeld op een andere, die niet op $e$ wordt afgebeeld. Een tegenspraak.

%Andersom, als $\sigma(e_{i})=e_{i}$, dan geldt automatisch dat $\sigma(e)=e$. Dat wil zeggen, de Galoisactie is compatibel met het morfisme $\mathcal{C}'_{0}\rightarrow{\mathcal{C}}$. 

%Verder, stel dat $\sigma(e)=e$, terwijl $\sigma(e_{i})\neq{e_{i}}$. Op dezelfde manier krijgen we dan weer een andere keten. <<{\bf{Samenvoegen met hierboven??}}>>
\end{proof}

%%%%%%% OUD
\begin{comment}
\begin{center}
{\bf{Het volgende is verwerkt:}}
\end{center}
\end{comment}
%%%%%%%%%% OUD

%%%%%% OUD
\begin{comment}
{\it{Een lijst van enkele makkelijke identiteiten:}}
\begin{enumerate}
\item $I_{e}=I_{e_{i}}=I_{e'_{i}}$.
\item Vertakkingscriterium voor inertiegroepen uit \cite{supertrop}.
\item Voor een willekeurige $H\subset{G}$, geldt voor $L^{H}$ dat $I'=I\cap{H}$ (en $D'=D\cap{H}$).  
\end{enumerate}
 \end{comment}
%%%%%%% OUD

\begin{prop}
Let $x_{i}\in\mathcal{C}_{0}$ be as above. Then
%Zij $e_{i}$ een "zijde" in $\mathcal{C}_{0}$. Dat wil zeggen, hij zit in het beeld van een zijde $e'_{i}$ in $\mathcal{C}$. Dan geldt dat 
\begin{equation}
I_{e}=I_{x_{i,i+1}}\subset{\prod_{i=0}^{n}I_{x_{i}}}.
\end{equation}  
\end{prop}
\begin{proof}
Consider the morphism $\mathcal{C}_{0}/{\prod_{i=0}^{n}I_{x_{i}}}\rightarrow{\mathcal{D}_{0}}$ and  suppose that it is ramified at the image of some $x_{i,i+1}$. Then it has to ramify in codimension one by purity of the branch locus. But the only possible candidates for this are the images of the $x_{i}$, a contradiction. This gives the desired result.
%Inderdaad $\mathcal{C}_{0}/{\prod_{i=0}^{n}I_{v_{i}}}\rightarrow{\mathcal{D_{0}}}$ is onvertakt bij het beeld van $e_{i}$ vanwege puurheid van de "branch locus". %tweede identiteit hierboven is het af.
\end{proof}

Let us quickly try the same argument to prove the other inclusion. Consider the morphism $\mathcal{C}_{0}/I_{e}\rightarrow{\mathcal{D}_{0}}$ and suppose that it is ramified at a vertical component $y_{i}$. From this point on, it is not directly evident how to predict the behavior of the corresponding connected edges. We will illustrate this in an example.

%I
 %try the same argument in reverse to prove the other inclusion, then we can't directly say anything about the behavior of the edge in terms of the inertia groups. Let us That is, consider the morphism $\mathcal{C}_{0}/I_{e}\rightarrow{\mathcal{D}_{0}}$ and suppose that it is ramified at a vertical component.
%Omgekeerd gaat het mis, stel dat $\mathcal{C}_{0}/I_{e}\rightarrow{\mathcal{D}_{0}}$ vertakt bij een verticale component. Dan kan er van alles gebeuren.
\begin{exa}
Let $A:=R[x,y]/(xy-\pi)$ and consider the covering given by the function field extension
%\begin{equation}
%xy=\pi,
%\end{equation}
%met 
\begin{equation}
K(x)\subset{K(x)[z]}/(z^2-\pi(x+1))=:L.
\end{equation}
The normalization of $A$ in $L$ is then vertically ramified at both components $\Gamma_{1}=Z(x)$ and $\Gamma_{2}=Z(y)$. Taking the tamely ramified extension $K\subset{K(\pi^{1/2})}$, we see that the normalization of $R'[x,y]/(xy-\pi)$ inside $L$ is now \'{e}tale above $(x,y,\pi^{1/2})$. We thus see that the inertia group can be unrelated to the inertia group of the edge after the extension.
%Deze vertakt in beide priemen, maar na de uitbreiding $K\subset{K(\pi^{1/2})}$ is hij \'{e}tale.
\end{exa}

We will now prove that $I_{e}\supset{\prod_{i=0}^{n}I_{x_{i}}}$. To do this, we will use {\it{Abhyankar's Lemma}}. 

\begin{lemma}\label{ramstruct2}{\bf{[Abhyankar's Lemma]}}
Let $X$ be a strictly Henselian local regular scheme of residue characteristic $p$, $D=\sum_{i=1}^{r}\text{div}(f_{i})$ a divisor with normal crossings on $X$ and $U=X-D$. Then every connected finite \'{e}tale covering of $U$ which is tamely ramified along $D$ is a quotient of a (tamely ramified) covering of the form
\begin{equation}
U'=U[T_{1},...,T_{r}]/(T_{1}^{n_{1}}-f_{1},...,T_{r}^{n_{r}}-f_{r}),
\end{equation}
where the $n_{i}$ are natural numbers prime to $p$.
%[Abhyankar's Lemma, Lemma 1.4. Brian Conrad's "Inertia groups and Fibers"].
\end{lemma}
\begin{proof}
See \cite[Theorem 1.2]{tamearithmetic} for the current formulation and \cite[Exp. XIII, 5.3., Page 316]{SGA1} for the proof.
\end{proof}

Let us consider this lemma for $y_{i,i+1}$ an intersection point in $\mathcal{D}_{0}$. Note that we have a natural morphism
\begin{equation}
\mathcal{O}_{\mathcal{D}_{0},y_{i,i+1}}\rightarrow{\mathcal{O}_{\mathcal{C}_{0},x_{i,i+1}}},
\end{equation}
giving rise to a morphism of completed rings
\begin{equation}
A:=\hat{\mathcal{O}}_{\mathcal{D}_{0},y_{i,i+1}}\rightarrow{\hat{\mathcal{O}}_{\mathcal{C}_{0},x_{i,i+1}}}.
\end{equation}
The ring $A$ is strictly Henselian, so we can apply Lemma \ref{ramstruct2}. We have
\begin{equation}
A\simeq{R[[x,y]]/(uv-\pi)}
\end{equation}
by assumption, and we thus obtain that $\hat{\mathcal{O}}_{\mathcal{D}_{0},y_{i,i+1}}\rightarrow{\hat{\mathcal{O}}_{\mathcal{C}_{0},x_{i,i+1}}}$ is a quotient of a Kummer covering of the form
\begin{equation}\label{Kummercov1}
A\rightarrow{}A[T_{1},T_{2}]/(T_{1}^{n_{1}}-u,T_{2}^{n_{2}}-v)
\end{equation}
for $n_{i}$ coprime to $p$.  
%In particular, using Lemma \ref{ramstruct2} for an ordinary double point with length $1$,

\begin{prop}\label{Inertiagroup1}
\begin{equation}
I_{e}={\prod_{i=0}^{n}I_{x_{i}}}.
\end{equation}
\end{prop}
\begin{proof}
We already proved that $I_{e}\subset{\prod_{i=0}^{n}I_{x_{i}}}$, so we will now prove the other inclusion. %Note that we have a natural morphism
%\begin{equation}
%\mathcal{C}_{0}/I_{e}\rightarrow{\mathcal{C}_{0}/\prod_{i=0}^{n}I_{x_{i}}}.
%\end{equation}
Let us first show that $I_{e}=I_{x_{1}}$. We first note that the natural morphism
\begin{equation}
\mathcal{C}_{0}/I_{x_{1}}\rightarrow{\mathcal{D}_{0}}
\end{equation}    
is \'{e}tale at the image of $x_{0,1}$. Indeed, if it were ramified, then it would be ramified in codimension one by purity of the branch locus. But it is already unramified at the image of both $x_{0}$ and $x_{1}$ (the first by the assumption on disjointly branched morphisms and the second by \cite[Proposition 3.7]{supertrop}), so we see that this is impossible.

We would now like to show that $\mathcal{C}_{0}/I_{x_{0,1}}\rightarrow{\mathcal{D}_{0}}$ is unramified at the image of $x_{1}$. Suppose that it is ramified. Since $x_{0}$ is unramified, we see that the associated morphism of completions from Equation \ref{Kummercov1} is of the form
\begin{equation}
A[T]/(T^{n}-v),
\end{equation}
where $v$ is a uniformizer for the local ring at $y_{1}$. But then a simple calculation shows that the length of the corresponding ring in $\mathcal{C}'_{0}/I_{x_{0,1}}$ would be strictly smaller. This contradicts the fact that $\mathcal{C}'_{0}/I_{x_{0,1}}=\mathcal{C}'_{0}/I_{x'_{0,1}}$ is \'{e}tale at the image of $x'_{0,1}$. We thus conclude that $I_{e}=I_{x_{1}}$.

We will now prove by rising induction that $I_{e}={\prod_{i=1}^{j}I_{x_{i}}}$ for every $j\leq{n}$. The case with $j=1$ was just treated. So assume that $I_{e}={\prod_{i=1}^{j-1}I_{x_{i}}}$. Consider the morphism
\begin{equation}
\mathcal{C}_{0}/{\prod_{i=1}^{j-1}I_{x_{i}}}\rightarrow{\mathcal{C}_{0}/{\prod_{i=1}^{j}I_{x_{i}}}}.
\end{equation}
Using the same reasoning as before, we see that $\mathcal{C}_{0}/{\prod_{i=1}^{j-1}I_{x_{i}}}\rightarrow{\mathcal{D}_{0}}$ is \'{e}tale at the image of $x_{j-1,j}$. The corresponding completed local ring %of the image of $x_{j-1,j}$ 
in $\mathcal{C}_{0}/{\prod_{i=1}^{j-1}I_{x_{i}}}$ is thus regular. Using Equation \ref{Kummercov1}, we see that the corresponding covering again must be of the form
\begin{equation}
A[T]/(T^{n}-v).
\end{equation}
Indeed, $\mathcal{C}_{0}/{\prod_{i=1}^{j-1}I_{x_{i}}}\rightarrow{\mathcal{C}_{0}/{\prod_{i=1}^{j}I_{x_{i}}}}$ is unramified at the image of $x_{j-1}$, so there is no other option. But then the corresponding length again decreases and we obtain another contradiction as in the $j=1$ case. By induction, we then conclude that $I_{e}={\prod_{i=0}^{n}I_{x_{i}}}$. 
%where

%$I_{e}={\prod_{i=0}^{n}I_{x_{i}}}
\end{proof}

We now set out to prove a formula for the inertia group $I_{x}$ in terms of the $I_{x_{i}}$. In the proof of Proposition \ref{Inertiagroup1}, we already saw that $I_{x}=I_{x_{1}}$. In general, the other inertia groups will be smaller. We first have the following
\begin{lemma}
Let $x$ be an intersection point in $\mathcal{C}$, $y$ its image in $\mathcal{D}$ and let $I_{x}$ be the corresponding inertia group. Then $I_{x}$ is cyclic.
\end{lemma}
\begin{proof}
This follows from $I_{x}=I_{x_{1}}$ and the fact that $I_{x_{1}}$ is cyclic (which is a result on tame Galois coverings of discrete valuation rings). For another proof, we note that
\begin{equation}
I_{x}=I_{\tilde{x}},
\end{equation}
where $\tilde{x}$ is the intersection point, considered as an element of a component $\Gamma'\subset{\mathcal{C}_{s}}$ and the inertia group is an inertia group for the induced Galois covering $\Gamma'\rightarrow{\Gamma}$. This equality follows from \cite[Proposition 3.9]{supertrop}. Since the local ring for $\tilde{x}$ in $\Gamma'$ is a discrete valuation ring, we again obtain that the inertia group is cyclic.  %can be given by relating the inertia group
\end{proof}

We now consider the cyclic abelian extension
\begin{equation}
\mathcal{C}_{0}\rightarrow{\mathcal{C}_{0}/I_{e}}.
\end{equation}

We note that $\mathcal{C}_{0}/I_{e}$ is again regular at the image of the chain induced by $e$. We now have

\begin{theorem}\label{InertProp2}
Let $I_{x_{i}}$ be as above. Then 
\begin{equation}
|I_{x_{i}}|=\dfrac{|I_{e}|}{\gcd(i,|I_{e}|)}.
\end{equation}
In particular, for $i$ such that $\gcd(i,|I_{e}|)=1$, we have that 
\begin{equation}
I_{x_{i}}=I_{e}.
\end{equation}
\end{theorem}
%the inertia group of the% generic point $y_{i}$ corresponding to $\Gamma'_{i}$. Then the following hold:
%\begin{enumerate}
%\item $I_{\Gamma'_{i}}\subseteq{I_{e'}}$.
 
%\end{enumerate} 
%In particular, for $i$ such that $\gcd(i,|I_{e}|)=1$, we have that 
%\begin{equation}
%I_{\Gamma_{i}}=I_{e}.
%\end{equation}
%\end{prop}
\begin{proof}
The corresponding extension of function fields for $\mathcal{C}_{0}\rightarrow{\mathcal{C}_{0}/I_{e}}$ is cyclic abelian, so we find by Kummer theory that it is given by an extension of the form
\begin{equation}
z^n=f
\end{equation}
for some $f\in{}K(\mathcal{C}_{0}/I_{e})$. Since $\mathcal{C}_{0}\rightarrow{\mathcal{C}_{0}/I_{e}}$ is unramified at $x_{0}$, we can assume that $v_{\tilde{x}_{0}}(f)=0$. Here $\tilde{x}_{0}$ is the image of $x_{0}$ in $\mathcal{C}_{0}/I_{e}$. Let $\tilde{x}$ be the image of $x$ in $\mathcal{C}/I_{e}$.  By \cite[Corollary 5.1]{tropabelian}, %(or the Poincar\'{e}-Lelong formula \cite[Theorem 5.15. (5)]{BPRa1}), 
we now find that 
\begin{equation}
v_{x_{i}}(f)=\text{Sl}_{\tilde{x}}(\psi))\cdot{i},
\end{equation}
where $\psi$ is the Laplacian of $f$ and $\text{Sl}_{\tilde{x}}(\psi)$ is the slope of $\psi$ along $\tilde{x}$ in $\mathcal{C}/I_{e}$.  Since $\tilde{x}$ is completely ramified in this extension, we find that $\text{gcd}( \text{Sl}_{\tilde{x}}(\psi),n)=1$. A simple calculation (using Newton polygons for instance, see \cite[Proposition 4.1]{supertrop}) now shows that the order of the inertia group is as stated in the Theorem.
%The first part follows from Lemma \ref{etalesubdivision1}. For the second part we will use the Laplacian and the Poincar\'{e}-Lelong formula.
%<<{\bf{Details invullen}}>>
\end{proof}

\subsection{The decomposition group of a vertex}

Let $\phi:\mathcal{C}\longrightarrow{\mathcal{D}}$ be a disjointly branched Galois morphism, with Galois group $G$. %In this section, we will assume that the generic fiber of $\mathcal{D}$ is isomorphic to $\mathbb{P}^{1}$. 
Let $\Gamma'$ be any irreducible component in the special fiber of $\mathcal{C}$ and let $\Gamma$ be its image in $\mathcal{D}$. %Let $D_{\Gamma'}$ be the decomposition group of the generic point of $\Gamma'$.
\begin{theorem}\label{DecompVert}
Suppose that the genus of $\Gamma$ is zero. Then
\begin{equation}
D_{\Gamma'}=\prod_{P\in{\Gamma'(k)}}I_{P}.%.=\prod_{P\in\Gamma'(k)}D_{P}.
\end{equation}
\end{theorem}

\begin{proof}
Let $y'$ be the generic point of $\Gamma'$. We factorize the morphism $\phi:\mathcal{C}\longrightarrow{\mathcal{D}}$ into $\mathcal{C}\longrightarrow{\mathcal{C}/D_{\Gamma'}}\longrightarrow{\mathcal{D}}$. Note that last morphism is {\it{"Nisnevich"}} at the image of $y$. That is, it is \'{e}tale and the induced map of residue fields is an isomorphism. In fact, $K(\mathcal{C}/D_{\Gamma'})$ is the largest among all such fields. Since the map on the residue fields is an isomorphism, we don't have any ramification and as such we find that $D_{\Gamma'}\supset{\prod_{P\in{\Gamma'(k)}}I_{P}}$. Note that this didn't use the condition on the genus of $\Gamma$. 

The induced morphism $\mathcal{C}/\prod_{P\in{\Gamma'(k)}}I_{P}\longrightarrow{\mathcal{C}/D_{\Gamma'}}$ is unramified above every point in the image of $\Gamma'$ in $\mathcal{C}/D_{\Gamma'}$. But this component has the same function field as $\Gamma$, which has genus zero. Since genus zero curves have no unramified coverings (by the Riemann-Hurwitz formula for instance), we find that $D_{\Gamma'}=\prod_{P\in{\Gamma'(k)}}I_{P}$, as desired. 
\end{proof}

\begin{rem}
Note that the condition on the genus of $\Gamma$ is indeed necessary. Take an elliptic curve $E$ with good reduction over $K$ and a corresponding model $\mathcal{E}$ with good reduction over $R$. Now take any unramified Galois covering of $E$ (which is in fact abelian, but we won't be needing this) with Galois group $G$ . Then the corresponding curve $E'$ again has genus $1$ and the corresponding intersection graph consists of only one vertex with weight $1$. We therefore see that $D_{\Gamma'}=G$, even though $I_{P}=(1)$ for every $P$. %Suppose that $E$ has a rational $2$-torsion point $P$. 
\end{rem}

 %as in Section \

\section{Covering data for $S_{3}$-coverings}\label{CoveringData}

In this section we specialize to a $3:1$ covering $C\longrightarrow\mathbb{P}^{1}$ that is not Galois. In other words, the Galois closure has degree six over $K(\mathbb{P}^{1})$. We will give the covering data in \emph{two ways}, where the first one uses the results from Section \ref{InertSection} and the second one the Laplacian on the intersection graph corresponding to the quadratic subfield $K(D)$. Both of these methods rely on a normalization result for $S_{3}$-coverings of discrete valuation rings. %, which we give first.   %so we'll start there. 

\subsection{Tame $S_{3}$-coverings of discrete valuation rings} \label{TameDisc}

Let $R$ be a discrete valuation ring with quotient field $K$, residue field $k$, uniformizer $\pi$ and valuation $v$. Note that the residue field $k$ is \emph{not} assumed to be algebraically closed for this section, since we'll be using the results here for valuations corresponding to components in the special fiber of a semistable model. We will denote the maximal ideal $(\pi)$ in $R$ by $\mathfrak{p}$. We will assume that $\text{char}(K)=0$, $\text{char}(k)>3$ . Furthermore, we will assume that $K$ contains a primitive third root % By Hensel's Lemma, we then find that $K$ contains a primitive root 
of unity $\zeta_{3}$ and a primitive fourth root of unity $\zeta_{4}$. The fourth root of unity is used to remove a minus sign in the formula for the discriminant (see Appendix \ref{Appendix1}), which is strictly speaking not necessary, but it makes the formulas somewhat nicer. %Note that this implies that $\sqrt{27}=3\sqrt{3}\in{K}$. %  We will also assume that $K$ contains a primitive third root of unity $\zeta_{3}$. 
The third root of unity will allow us to use Kummer theory for abelian coverings of degree $3$. %and $2$. %For purely aesthetic reasons, we also assume that $\sqrt{27}\in{K}$ (see the formulas in Appendix \ref{Appendix1})%To make the upcoming formulas somewhat simpler, we will also assume that $\sqrt{27}\in{K}$. 

%Let $S$ be a discrete valuation ring extending $R$ with field of fractions $L$. 
Let $L$ be a degree $3$ extension of $K$ such that $\overline{L}/K$ has Galois group $S_{3}$. This is equivalent to the discriminant of $L/K$ not being a square in $K$.
%Suppose that $K\subset{L}$ is of degree $3$. 
After a translation, $L$ is given by an equation of the form
\begin{equation}\label{MainEq1}
z^3+p\cdot{z}+q=0,
\end{equation}
where $p,q\in{K}$.
 Let $B$ be the normalization of $R$ in $\overline{L}$ and let $\mathfrak{q}$ be any prime lying above $\mathfrak{p}$. %that is not Galois.  By scaling, we obtain a monic equation of the same form with $p,q\in{R}$. We will assume that $L\supset{K}$ is not Galois. We now consider the algebra
%\begin{equation*}
%B:=R[w]/(w^6+2\sqrt{27}q\cdot{w^3}-4p^3).
%\end{equation*} 
%Over $K$, this just gives the {\bf{Galois closure}} of the extension $K\subseteq{L}$. Note that it contains the finite $R$-subalgebra $A:=R[y]/(y^2-(\Delta))$, with $\Delta:=4p^3+27q^2$. Over $K$, this gives the quadratic subfield of the Galois closure. In terms of $y$, we then have $w^3=y\pm\sqrt{27}q$.\\
%We now wish to find the normalization of $B$ in $\overline{L}$. We would futhermore like to know the order of the inertia group $|I_{Q}|$ of any prime lying above $(\pi)$. 
We would now like to know the \emph{inertia group} of $\mathfrak{q}$. % Note that the inertia group of any other prime $\sigma(\mathfrak{q})$ lying above $\mathfrak{p}$ is just $\sigma{}I_{\mathfrak{q}}\sigma^{-1}$. 
We will content ourselves with knowing $|I_{\mathfrak{q}}|$, which suffices for practical purposes.  %related to that of $\mathfrak{q}$ by a conjugated subgroup. %any prime $\mathfrak{q}\in\text{Spec}(B)$ lying above $\mathfrak{p}=(\pi)\in\text{Spec}(R)$. 

Let us state the relevant results here and defer the actual proofs and computations to Appendix \ref{Normalizations}.  %We consider three cases:
%\begin{enumerate}
%\item $3v(p)>2v(q)$,
%\item $3v(p)<2v(q)$,
%\item $3v(p)=2v(q)$.
%\end{enumerate}

%We then have the following
\begin{prop}\label{InertS3}
\begin{enumerate}
\item Suppose that $3v(p)>2v(q)$. Then
\begin{eqnarray*}
|I_{\mathfrak{q}}|=3 &\iff & 3\nmid{v(q)},\\
|I_{\mathfrak{q}}|=1 & \iff & 3\mid{v(q)}.
\end{eqnarray*} 
\item Suppose that $3v(p)<2v(q)$. Then
\begin{eqnarray*}
|I_{\mathfrak{q}}|=2 &\iff & 2\nmid{v(p)},\\
|I_{\mathfrak{q}}|=1 & \iff & 2\mid{v(p)}.
\end{eqnarray*} 
\item Suppose that $3v(p)=2v(q)$. Then 
\begin{eqnarray*}
|I_{\mathfrak{q}}|=2 &\iff & 2\nmid{v(\Delta)},\\
|I_{\mathfrak{q}}|=1 & \iff & 2\mid{v(\Delta)}.
\end{eqnarray*} 
\end{enumerate}
\end{prop}
\begin{proof}
See Appendix \ref{Normalizations}. 
\end{proof}

\subsection{Covering data using continuity of inertia groups}\label{InertTechnique}

In this section, we give the covering data for the morphism of intersection graphs $\Sigma(\overline{\mathcal{C}})\rightarrow\Sigma({\mathcal{D}_{S}})$ associated to a disjointly branched morphism $\overline{\mathcal{C}}\rightarrow{\mathcal{D}_{S}}$ for the $S_{3}$-covering $\phi: \overline{C}\rightarrow{\mathbb{P}^{1}}$. Here $S$ is a subset of $\mathbb{P}^{1}(K)$ that contains the branch locus of $\overline{\phi}$. Before we give the covering data, we will first give an \emph{explicit} $S\subset{\mathbb{P}^{1}(K)}$ that contains the branch locus, giving rise to a separating model $\mathcal{D}_{S}$ as explained in Appendix \ref{Appendix2}. After that, we will use Proposition \ref{InertS3} and Theorem \ref{InertProp2} to give the {\it{covering data}}. That is, we will give $|D_{x}|$, where $x$ corresponds to an edge or vertex in $\Sigma(\overline{\mathcal{C}})$.

For any $z\in\mathbb{P}^{1}$, let $v_{z}$ be the corresponding valuation of the function field $K(x)$. Consider the set $S:=\text{Supp}(p,q,\Delta)\subset{\mathbb{P}^{1}}$.  In terms of valuations, we then find that $z\in{S}$ if and only $v_{z}$ is nontrivial on $p,q$ or $\Delta$. %Here $v_{z}$ is the natural valuation of the function field $K(x)$ associated to $z$.  % Let us translate this definition into the language of "generic valuations" of $K(x)$. The valuations in $S$ are then exactly those valuations of $K(x)$ that are trivial on $K$ and nontrivial on $p,q$ or $\Delta$. 
From now on, we assume that  %$%\subset{\mathbb{P}^{1}(K)}$, 
$S\subset{\mathbb{P}^{1}(K)}$ (otherwise, we take a finite extension $K'$ of $K$ and set $K:=K'$). Let $B_{\overline{\phi}}$ be the branch locus of the morphism $\overline{\phi}:\overline{C}\rightarrow{\mathbb{P}^{1}}$. We then have
\begin{lemma}\label{BranchLocus1}
\begin{equation}
B_{\overline{\phi}}\subseteq{S}. 
\end{equation}
\end{lemma}
\begin{proof}
This follows from Proposition \ref{InertS3} and the characterization of $S$ in terms of valuations given before the lemma. %we see that the only candidates for branch points are given by those $z\in\mathbb{P}^{1}$ the valuations $v_{z}$ of $K(x)$ with $v(K)=0$ such that either %only places of $K(x)$ where any ramification is possible, are given by the valuations with 
%$v(p)\neq{0}$, $v(q)\neq{0}$ or $v(\Delta)\neq{0}$. This is exactly the support of $(p,q,\Delta)$, as desired. 
\end{proof}

For $S=\text{Supp}(p,q,\Delta)$, we now take a model $\mathcal{D}_{S}$ of $\mathbb{P}^{1}$ such that the closure of $S$ in $\mathcal{D}_{S}$ consists of disjoint smooth sections over $R$. See Appendix \ref{Appendix2} for the construction of $\mathcal{D}_{S}$ and its corresponding intersection graph (which is also known as the {\it{tropical separating tree}}). Note that the morphism $\overline{\mathcal{C}}_{0}\rightarrow{\mathcal{D}_{S}}$ obtained by normalizing might not be disjointly branched, as the generic points of components in the special fiber of $\mathcal{D}_{S}$ can ramify. By Proposition \ref{InertS3}, we see that as soon as we know the valuations $v_{\Gamma}(p), v_{\Gamma}(q),v_{\Gamma}(\Delta)$ for a component $\Gamma$, we know what tamely ramified extension we have to take to make this morphism disjointly branched. %Assuming that these are known, w
We then obtain a disjointly branched morphism $\overline{\mathcal{C}}\rightarrow{\mathcal{D}'_{S}}$, where $\mathcal{D}'_{S}=\mathcal{D}_{S}\times{\text{Spec}(R')}$ for $R'$ a discrete valuation ring in $K'$ dominating $R$.

We now take a regular subdivision $\mathcal{D}''_{S}$ of $\mathcal{D}'_{S}$ in any edge $e$ and obtain a subdivision of the intersection graph $\Sigma(\mathcal{D}'_{S})$. As before, we note that the corresponding normalization of this model $\mathcal{D}''_{S}$ in $K(\overline{C})$ can then be vertically ramified over $\mathcal{D}''_{S}$. % and we can take a tamely ramified extension to remedy this. We \emph{don't} do this and instead just note
The good news here is that the inertia groups of the new components are directly related to the inertia group of the original edge $e$ by Theorem \ref{InertProp2}.   %Even though it seems like we're getting nowhere with this process, we see by Theorem \ref{InertProp2} that the inertia groups of the new components are directly related to the inertia group of the original edge $e$.  %by Theorem \ref{InertProp2}.  %but the good news %This process of taking a finite extension and then taking a regular subdivision can be continued indefinitely but the good news %continue this process %of "\emph{take a finite extension, take a regular subdivision, take a finite extension...}" indefinitely. The good news however 
%is that the inertia groups of the new components are directly related to the inertia group of the original edge $e$ by Theorem \ref{InertProp2}. 
We thus see by Proposition \ref{InertS3} that if we know the set $v_{\Gamma}(p), v_{\Gamma}(q),v_{\Gamma}(\Delta)$ for \emph{any} component in \emph{any} subdivision of our original intersection graph, then we know the inertia group of the original edge.   % By Theorem \ref{Inert

% Furthermore, if we know these valuations for \emph{any} component in a regular subdivision of $\mathcal{D}_{S}$ in an edge, we know the order of the inertia group of that edge by Theorem \ref{InertProp2}. We thus see that if we know the set 
%$v_{\Gamma}(p), v_{\Gamma}(q),v_{\Gamma}(\Delta)$ for any component, then we know the splitting behavior of any edge in the intersection graph $\Sigma(\mathcal{D}_{S})$.

We can find these valuations $v_{\Gamma}(p), v_{\Gamma}(q),v_{\Gamma}(\Delta)$ directly using the Laplacian operator from Section \ref{Laplacian1}. %One way to look at this, is that the Laplacian takes  expresses this 
This is the content of the next theorem.  
\begin{theorem}\label{MainThmVert}
Let $f\in\mathcal{K(D)}$, where $\mathcal{D}$ is assumed to be regular. Let $\rho(\text{div}_{\eta}(f))$ be the induced principal divisor of $f$ on the intersection graph $G$ of $\mathcal{D}$. %Fix a component $\Gamma$. 
Write
\begin{equation*}
\Delta(\phi)=\rho(\text{div}(f))
\end{equation*}
for some $\phi:\mathbb{Z}^{V}\longrightarrow{\mathbb{Z}}$, where $\Delta$ is the Laplacian operator. Choose $\phi$ such that $\phi(\Gamma)=0$. Then the unique vertical divisor corresponding to $\text{div}_{\eta}(f)$ with $V_{f^{\Gamma}}(\Gamma)=0$ is given by
\begin{equation}\label{ExplVert2}
V_{f^{\Gamma}}=\sum_{i}\phi(\Gamma_{i})\cdot{\Gamma_{i}}.
\end{equation} 
\end{theorem}

\begin{proof}
See \cite[Theorem 5.2]{tropabelian}. 
\end{proof}
\begin{rem}\label{Remark1}
The valuation of $f$ at a component $\Gamma$ is exactly the coefficient in the vertical divisor corresponding to $\text{div}(f)$. We thus see that the above theorem gives the valuation, as soon as we know the valuation of $f$ at a single component. For $K(\mathcal{D})=K(x)$, this is quite easy: we take the valuation $v_{\Gamma_{0}}$ corresponding to the prime ideal $\mathfrak{p}=(\pi)\subset{R[x]}$. To be explicit, we write $f=\pi^{k}g$ (with $g\notin\mathfrak{p}$) and find $v_{\Gamma_{0}}(f)=k$. In other words, this valuation just measures the power of $\pi$ in $f$. 
\end{rem}
\begin{rem}
If we take any base change of the form $K\subset{K(\pi^{1/n})}$, then the corresponding Laplacians for $p,q$ and $\Delta$ are scaled by a factor $n$ (at least, if we normalize our valuation such that $v(\pi^{1/n})=1$).  
\end{rem}

We now summarize the above method for finding the covering data for a disjointly branched $S_{3}$-covering $\mathcal{C}\rightarrow{\mathcal{D}_{S}}$ corresponding to the $S_{3}$-covering $\overline{C}\rightarrow{\mathbb{P}^{1}}$. 

\begin{center}
{\bf{Algorithm for the covering data using continuity of inertia groups}}
\end{center}
\begin{enumerate}
\item Construct the tropical separating tree for the set $S=\text{Supp}(p,q,\Delta)\subset{\mathbb{P}^{1}(K)}$. %Z(p)\cup{Z(q)}\cup{Z(\Delta)}\subset{\mathbb{P}^{1}(K)}$.
\item For every root (and pole) $\alpha$ of $p$, $q$ and $\Delta$, determine $v_{\alpha}(p)$, $v_{\alpha}(q)$ and $v_{\alpha}(\Delta)$.
\item Find $v_{\Gamma_{0}}(p)$, $v_{\Gamma_{0}}(q)$ and $v_{\Gamma_{0}}(\Delta)$, as explained in Remark \ref{Remark1}. %the discussion before the Algorithm. % paragraph before Example \ref{Examplecurve}. 
\item Determine the corresponding Laplacians of $p$, $q$ and $\Delta$. 
%\item Take a finite extension $K'$ to eliminate the vertical ramification. This scales the Laplacians found above by a factor $n$, where $n$ is the degree of $K'/K$. 
\item Use Theorems \ref{InertProp2}, \ref{DecompVert} and Proposition \ref{InertS3} to determine the covering data.
\end{enumerate}
  %We subdivide the edge $e$ to obtain a local de

\begin{exa}\label{Examplecurve}
Let $C$ be the curve given by the equation
\begin{equation}
f(z)=z^3+p\cdot{z}+q=0
\end{equation}
for $p=x^3$ and $q=x^3+\pi^3$.
That is, we consider the field extension
\begin{equation}
K(x)\subset{K(x)[z]/(f(z))}
\end{equation}
and let $C\longrightarrow{\mathbb{P}^{1}}$ be the corresponding morphism of smooth curves. 
Let us find the divisors of $p$, $q$ and $\Delta=4p^3+27q^2=4x^9+27(x^3+\pi^3)^2$. 

\begin{figure}[h!]
\centering
\includegraphics[scale=0.3]{{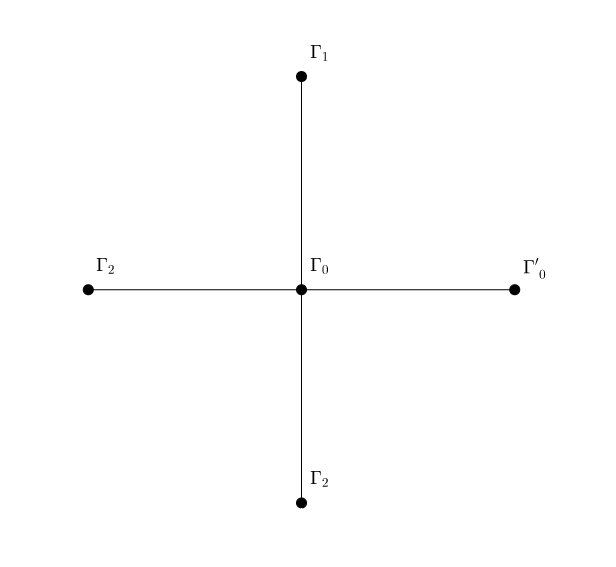}}
\caption{\label{1eplaatje} {\it{The tropical separating tree in Example \ref{Examplecurve}.}}}%{\it{The Laplacian function $\phi$ of $f$, as in Example \ref{Exa3torsgen2}. The $e_{i}$ denote the three edges between the two vertices.}}} %The dual intersection graph for every reduction type as in Theorem \ref{ThmRedType}.}
\end{figure}

Let
\begin{align*}
P_{0}&=(0),\\
P_{i}&=(-\zeta^{i}_{3}\cdot{\pi}),
\end{align*}
for $i\in\{1,2,3\}$ and $\zeta_{3}$ a primitive third root of unity. Then $(p)=3P_{0}-3(\infty)$ and $(q)=P_{1}+P_{2}+P_{3}-3(\infty)$. 
We now take the tropical separating tree with five vertices, marked as in Figure \ref{1eplaatje}. %by $\Gamma_{i}$ for $i\in\{0,1,2,3,4\}$ as in Image ...

%%%PLAATJE
%\begin{comment}
%\begin{tikzpicture}[line cap=round,line join=round,>=triangle 45,x=1cm,y=1cm]\draw[->,color=black] (-3.691383542643555,0) -- (13.37384636659718,0);\foreach \x in {-3,-2,-1,1,2,3,4,5,6,7,8,9,10,11,12,13}\draw[shift={(\x,0)},color=black] (0pt,2pt) -- (0pt,-2pt) node[below] {\footnotesize $\x$};\draw[->,color=black] (0,-7.978079912007616) -- (0,5.674104015384935);\foreach \y in {-7,-6,-5,-4,-3,-2,-1,1,2,3,4,5}\draw[shift={(0,\y)},color=black] (2pt,0pt) -- (-2pt,0pt) node[left] {\footnotesize $\y$};\draw[color=black] (0pt,-10pt) node[right] {\footnotesize $0$};\clip(-3.691383542643555,-7.978079912007616) rectangle (13.37384636659718,5.674104015384935);\draw (5,4)-- (5,3);\draw (5,2)-- (5,3);\draw (5,3)-- (6,3);\draw (4,3)-- (5,3);\begin{scriptsize}\draw [fill=qqqqff] (5,3) circle (2.5pt);\draw[color=qqqqff] (5.77386606980741,3.481420404953428) node {$$\Gamma_{0}$$};\draw [fill=qqqqff] (5,4) circle (2.5pt);\draw[color=qqqqff] (5.77386606980741,4.493428225152585) node {$$\Gamma_{1}$$};\draw [fill=qqqqff] (5,2) circle (2.5pt);\draw[color=qqqqff] (5.77386606980741,2.4892558753464114) node {$$\Gamma_{2}$$};\draw [fill=qqqqff] (6,3) circle (2.5pt);\draw[color=qqqqff] (6.8454037617829915,3.481420404953428) node {$$\Gamma'_{0}$$};\draw [fill=qqqqff] (4,3) circle (2.5pt);\draw[color=qqqqff] (4.78170154020039,3.481420404953428) node {$$\Gamma_{2}$$};\end{scriptsize}\end{tikzpicture}
%%%%PLAATJE
%\end{comment}

We then have the following \emph{tropical} divisors:
\begin{align}
\rho((p))&=3\Gamma_{0}-3\Gamma'_{0},\\
\rho((q))&=\Gamma_{1}+\Gamma_{2}+\Gamma_{3}-3\Gamma'_{0}.
\end{align}
We quickly see that $p$ and $q$ contain no factors of $\pi$, so $v_{\Gamma'_{0}}(p)=v_{\Gamma'_{0}}(q)=0$. The Laplacians are then given by Figure \ref{Laplacianen1}. Note that the Laplacian $\phi_{p}$ is the same on every segment $e_{i}:=\Gamma_{i}\Gamma_{0}$ and likewise for $\phi_{q}$. We see that $\phi_{p}$ has slope zero on the $e_{i}$ and slope $3$ on $\Gamma_{0}\Gamma'_{0}$. Furthermore, we see that $\phi_{q}$ has slope $1$ on every $e_{i}$ and slope $3$ on $\Gamma_{0}\Gamma'_{0}$.  %That is: $\phi_{p}|_{e_{1}}=\phi_{p}|_{e_{2}}=\phi_{p}|_{e_{3}}$. The same identity also holds for $\phi_{q}$. 

%Note that for every $i$ the Laplacian of is the same on every line segment $\Gamma_{i}\Gamma_{0}$ for every $i$, for both $p$ and $q$.

\begin{figure}
\begin{subfigure}[b]{.45\textwidth}
  \centering
  \includegraphics[width=.5\linewidth, height=5cm]{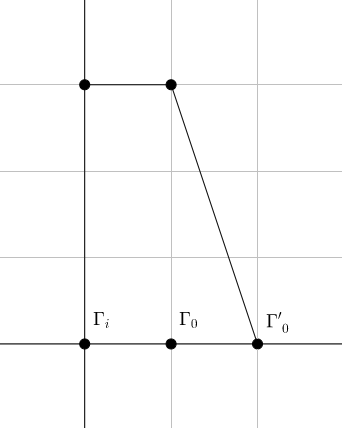}
  \caption{The Laplacian $\phi_{p}$ of $p$.}%The Laplacian of $p$.}
  \label{2eplaatje}
\end{subfigure}%
\begin{subfigure}[b]{.45\textwidth}
  \centering
  \includegraphics[width=.7\linewidth, height=5cm]{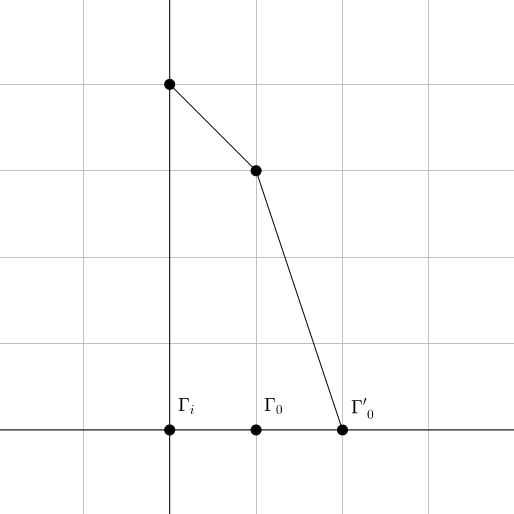}
  \caption{The Laplacian $\phi_{q}$ of $q$.}
  \label{3eplaatje}
\end{subfigure}
\caption{The Laplacians for Example \ref{Examplecurve}.}
\label{Laplacianen1}
\end{figure}
%\centering
%\includegraphics[scale=0.4]{{Graph2.png}}
%\caption{\label{2eplaatje} {\it{The Laplacian of $p$}}}%{\it{The Laplacian function $\phi$ of $f$, as in Example \ref{Exa3torsgen2}. The $e_{i}$ denote the three edges between the two vertices.}}} %The dual intersection graph for every reduction type as in Theorem \ref{ThmRedType}.}
%\end{figure}

\begin{figure}[h!]
\centering
\includegraphics[scale=0.3]{{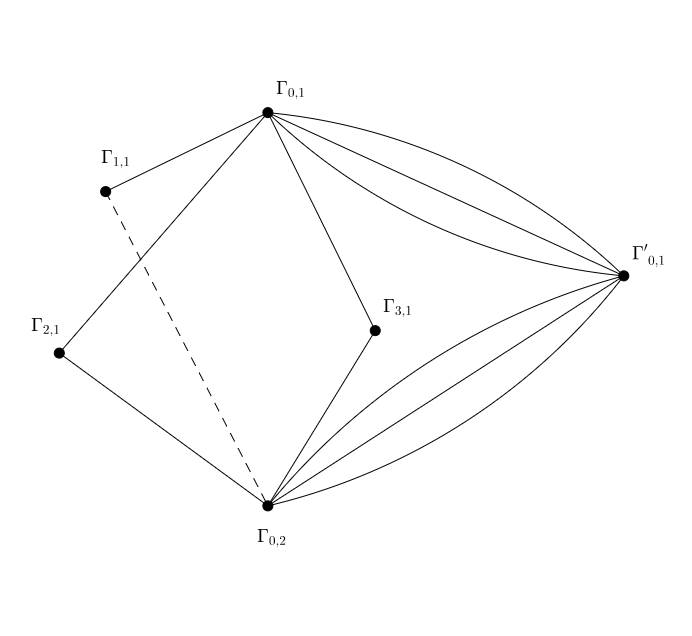}}
\caption{\label{5eplaatje} The intersection graph of the Galois closure in Example \ref{Examplecurve}.}
\end{figure}

For every vertex, we then have $3v_{\Gamma}(p)>2v_{\Gamma}(q)$, so we are in Case (I) of Theorem \ref{InertS3}. We then see that $\Gamma_{0}$ and $\Gamma'_{0}$ are unramified and every $\Gamma_{i}$ is ramified of order $3$. So we first take the tamely ramified extension of order three: $K\subset{K(\pi^{1/3})}$. If we now take a regular model after the base change, the Laplacians of both $p$ and $q$ will be scaled by a factor three (if we assume that our new valuation is normalized such that $v(\pi^{1/3})=1$). Furthermore, for this regular model we see that every edge gives rise to two new components. % is subdivided into four components. 
By Theorem \ref{InertProp2}, we have that the inertia groups of the new components on the subdivisions in fact give the inertia groups of the original edges.

Let us illustrate this in more detail. For instance, if we take the edge $e_{i}$, then after taking the base change we obtain four components: $\Gamma_{i}, v_{i,1},v_{i,2}$ and $\Gamma_{0}$, where $v_{i,1}$ and $v_{i,2}$ are new. We then find that the new Laplacian $\tilde{\phi}_{q}$ has 
\begin{align*}
\tilde{\phi}_{q}(\Gamma_{0})&=9, \\
\tilde{\phi}_{q}(v_{i,2})&=10, \\
\tilde{\phi}_{q}(v_{i,1})&=11,\\
\tilde{\phi}_{q}(\Gamma_{i})&=12.
\end{align*}

Again, using Theorem \ref{InertS3}, we find that $|I_{v_{i,2}}|=3$. By Theorem \ref{InertProp2}, we see that $|I_{e_{i}}|=3$. In other words, there are two edges lying above every $e_{i}$.
%$\Gamm_{0} %finite extensio
%For the edges $e_{i}=\Gamma_{i}\Gamma_{0}$, we see that $v(q)$ is not divisible by $3$, so we have $|I_{e'_{i}}|=3$ for an edge $e'_{i}$ lying above $e_{i}$. In other words, there are two edges lying above them. 
For $e_{0}=\Gamma_{0}\Gamma'_{0}$, using the same procedure as before, we see that there are six edges lying above $e_{0}$. Using Theorem \ref{DecompVert}, we see that there are two vertices lying above $\Gamma_{0}$. One then quickly finds that there is only one covering graph $\Sigma(\overline{\mathcal{C}})$ satisfying these conditions. It is given by Figure \ref{5eplaatje}.

We note that the genera of $\Gamma_{0,1}$, $\Gamma_{0,2}$ and $\Gamma'_{0,1}$ are one, whereas the genera of the other components are zero. This can be found using the Riemann-Hurwitz formula. 
\begin{figure}[h!]
\centering
\includegraphics[scale=0.3]{{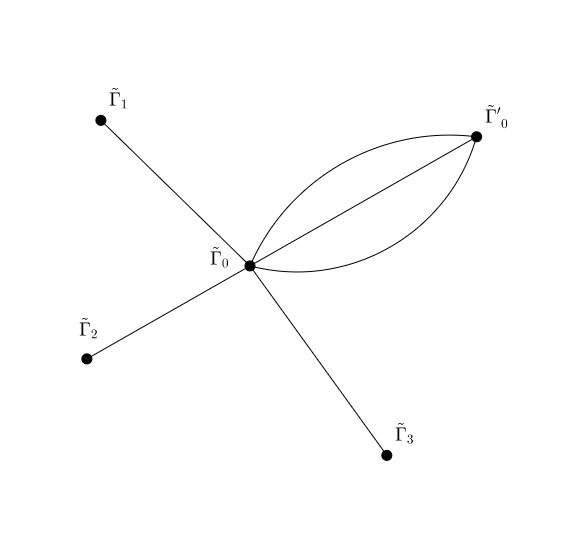}}
\caption{\label{6eplaatje} The intersection graph of the quotient under the subgroup of order two in Example \ref{Examplecurve}.}
\end{figure}

Taking the quotient under the subgroup of order two corresponding to the curve $C$, we then obtain the intersection graph of $\mathcal{C}$. It is given in Figure \ref{6eplaatje}. Note that the component labeled by $\tilde{\Gamma}_{0}$ has genus $1$, whereas $\tilde{\Gamma}'_{0}$ has genus $0$. The other three components don't contribute to the Berkovich skeleton. The entire Galois lattice, including all the intermediate intersection graphs but excluding the leaves, can now be found in Figure \ref{21eplaatje}. 

\begin{figure}[h!]
\centering
\includegraphics[scale=0.45]{{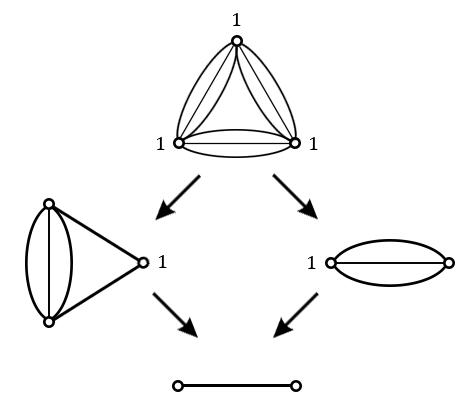}}
\caption{\label{21eplaatje} The full Galois lattice of intersection graphs in Example \ref{Examplecurve}. The leaves are omitted. }
\end{figure}

\end{exa} 

Using a different technique, the same result was obtained in \cite[Example 8.2]{tropabelian}. In fact, Section \ref{QuadraticSubfieldTechnique} will highlight the technique used there.  

%Note that this example was also given in \cite[Example 8.2.]{tropabelian}, where the same Berkovich skeleton was acquired, but with a different technique. <<Nadruk ergens anders leggen>>

\subsection{Covering data using the quadratic subfield}\label{QuadraticSubfieldTechnique}

%Let $\overline{\mathcal{C}}\longrightarrow{\mathcal{D}_{0}}$ be the disjointly branched morphism corresponding to the $S_{3}$-covering $\overline{C}\rightarrow{\mathbb{P}^{1}}$. In this section, we give the covering data corresponding to the morphism of intersection graphs $\Sigma(\overline{\mathcal{C}})\rightarrow{\Sigma(\mathcal{D})}$ in terms of the subcover $\overline{C}\rightarrow{D}\rightarrow{\mathbb{P}^{1}}$ and the corresponding morphisms of models
%\begin{equation}
%\overline{\mathcal{C}}\rightarrow{\mathcal{D}}\rightarrow{\mathcal{D}_{0}}.
%\end{equation}

The degree three covering $\phi:C\longrightarrow{\mathbb{P}^{1}}$ can be represented on the level of function fields as
\begin{equation}
z^3+p\cdot{}z+q=0,
\end{equation}
where $p$ and $q$ are polynomials over $K$. By our initial assumption on $\phi$, we find that the Galois closure %If $\phi$ is not Galois, then the Galois closure 
contains a quadratic subfield $K(D)$, corresponding to a smooth curve $D$. On the level of function fields, this is given as
\begin{equation}
K(x)\subset{}K(x)[y]/(y^2-\Delta),
\end{equation}
where $\Delta=4p^3+27q^2$.  %The Galois closu

The corresponding degree three morphism $\overline{C}\rightarrow{D}$ can then be represented by
\begin{equation}
K(D)\subset{K(D)[w]/(w^3-(y-\sqrt{27}q))}.
\end{equation}

We thus see that the field extension $K(\overline{C})\supset{K(\mathbb{P}^{1})}$ has been subdivided into two abelian parts: $K(\overline{C})\supset{K(D)}$ of degree $3$ and $K(D)\supset{K(\mathbb{P}^{1})}$ of degree $2$. See Appendix \ref{Appendix1} for some background material regarding these equations.

Consider a model $\mathcal{D}_{S}$ of $\mathbb{P}^{1}$ such that the closure of $S:=\text{Supp}(p,q,\Delta)$ %\{Z(p),Z(q),Z(\Delta)\}$ in $\mathcal{D}_{0}$
 is separated in $\mathcal{D}_{S}$. %See Appendix \ref{Appendix2} for a review of how to obtain such a model. 
 We find by Lemma \ref{BranchLocus1} that the branch locus of $\overline{\phi}$ is contained in $S$. Over some finite extension $K'\supset{K}$, we thus obtain disjointly branched morphisms $\overline{\mathcal{C}}\rightarrow{\mathcal{D}}\rightarrow{\mathcal{D}_{S}}$  such that the base change to the generic fiber is $\overline{C}\rightarrow{D}\rightarrow{\mathbb{P}^{1}}$.   %associated to $\overline{\phi}$. %,  giving the morphisms $\overline{C}\rightarrow{D}\rightarrow{\mathbb{P}^{1}}$ on the generic fiber.% over some finite extension $K'$ of $K$. 
 We won't worry about this finite extension in this section and just take $K:=K'$.   %$\mathcal{D}_{S. 
 We first calculate the intersection graph of the intermediate model $\mathcal{D}$. % which is the normalization of an appropriate base change of $\mathcal{D}_{S,0}$ in $K'(D)$ for a finite extension $K'$. 
 The covering $D\rightarrow{\mathbb{P}^{1}}$ is hyperelliptic, so the techniques of \cite{supertrop} and \cite{tropabelian} are directly applicable. We note that this step does not require any twisting data, since we are dealing with an abelian covering of a tree.

We now consider the divisor of the function $y-\sqrt{27}q\in{K(D)}$. This can be given explicitly in terms of the zero divisors of $p,q$ and $\Delta$. Since calculating divisors is a matter of normalizing, the reader will probably not be surprised that there are again three cases. The result is as follows, where we again defer the proof to Appendix \ref{Normalizations}.
\begin{prop}\label{DivisorDegree3}
Let $y-\sqrt{27}q\in{K(D)}$ be as above. Let $z\in\mathbb{P}^{1}(K)$ and denote by $v_{z}$ the corresponding valuation of $K(x)$.
\begin{enumerate}
\item Suppose that $3v_{z}(p)>2v_{z}(q)$. There are then two points $Q_{1}$ and $Q_{2}$ lying above $P$ in $D$. We then have
\begin{align*}
v_{Q_{1}}(y-\sqrt{27}q)&=3v_{z}(p)-v_{z}(q),\\
v_{Q_{2}}(y-\sqrt{27}q)&=v_{z}(q).
\end{align*}
\item Suppose that $3v_{z}(p)<2v_{z}(q)$. If $2|v_{z}(p)$, then there are two points $Q_{1}$ and $Q_{2}$ lying above $P$ in $D$. We have
\begin{equation}
v_{Q_{i}}(y-\sqrt{27}q)=3v_{z}(p)/2.
%v_{Q_{2}}(y-\sqrt{27}q)&=3v_{P}(p)/2.
\end{equation}
If $2\nmid{v_{z}(p)}$, then there is only one point $Q$ in $D$ lying above $P$. We then have
\begin{equation}
v_{Q}(y-\sqrt{27}q)=3v_{z}(p).
\end{equation}
\item Suppose that $3v_{z}(p)=2v_{z}(q)$. If $2|v_{P}(\Delta)$, then there are two points $Q_{1}$ and $Q_{2}$ lying above $P$. We have
 \begin{equation}
 v_{Q_{i}}(y-\sqrt{27}q)=v_{z}(q).
 \end{equation}
 If $2\nmid{v_{z}(\Delta)}$, then there is only one point $Q$ lying above $P$. We have
 \begin{equation}
 v_{Q}(y-\sqrt{27}q)=2v_{z}(q).
 \end{equation}
\end{enumerate}
\end{prop}
\begin{proof}
See Appendix \ref{Normalizations}. 
\end{proof}

By calculating the reduction of every $Q_{i}$ in $\mathcal{D}$, one then obtains the tropical divisor $\rho(\text{div}(y-\sqrt{27}q))$. This is a principal divisor in $\text{Div}^{0}(\Sigma(\mathcal{D}))$, so we can write 
\begin{equation}
\Delta(\phi)=\rho(\text{div}(y-\sqrt{27}q))
\end{equation}
for some $\phi$ in $\mathcal{M}(\Sigma(\mathcal{D}))$. The covering data for an edge $e\in\Sigma(\mathcal{D})$ is then obtained as follows:
\begin{prop}\label{QuadraticSubfield}
Let $\Sigma(\mathcal{D})$ be the intersection graph of $\mathcal{D}$ and let $\phi$ be such that 
\begin{equation}
\Delta(\phi)=\rho(\text{div}(y-\sqrt{27}q)).
\end{equation} 
Let $e$ be an edge in $\Sigma(\mathcal{D})$ and let $\delta_{e}(\phi)$ be the absolute value of the slope of $\phi$ along $e$. 
Then the following hold:
\begin{enumerate}
\item There are three edges above $e$ if and only if $3|\delta_{e}(\phi)$. 
\item There is one edge above $e$ if and only if $3\nmid{\delta_{e}(\phi)}$.
\end{enumerate}
Furthermore, there are three vertices above a vertex $v$ with corresponding component $\Gamma$ if and only if the reduction of %the modified divisor $\text{div}
$(y-\sqrt{27}q)^{\Gamma}$ is a cube. 
\end{prop}
\begin{proof}
The statements regarding the splitting behavior of the edge are recorded in \cite[Proposition 5.7]{tropabelian} and \cite[Proposition 4.1]{supertrop}. The statement about the vertices is Lemma \ref{FactorizationVertexSplit}.  %(where we apologize to the reader for referencing to a 
\end{proof}

%\begin{exa}
%Zelfde voorbeeld als in de vorige sectie, nu met deze methode.
%\end{exa}

We now summarize the above method for finding the covering data for a tame $S_{3}$-covering $\mathcal{C}\rightarrow{\mathcal{D}_{\mathbb{P}^{1}}}$ using the quadratic subfield $K(D)$. 

\begin{center}
{\bf{Algorithm for the covering data using the quadratic subfield}}
\end{center}
\begin{enumerate}
\item Construct the tropical separating tree corresponding to $S=\text{Supp}(p,q,\Delta)$. %Z(p)\cup{Z(q)}\cup{Z(\Delta)}$.
\item Calculate the intersection graph $\Sigma(D)$ using the disjointly branched morphism $\mathcal{D}\rightarrow{\mathcal{D}_{\mathbb{P}^{1},S}}$.
\item Calculate the Laplacian of $y-\sqrt{27}q$ on $\Sigma(D)$ using Proposition \ref{DivisorDegree3}.
\item Calculate the covering data for the edges using Proposition \ref{QuadraticSubfield}.
\item If $\text{div}(\overline{f^{\Gamma}})\equiv{0}\mod{3}$, determine if it is a cube in $k(\Gamma)$. (This requires additional computations on the residue fields of the components of $\mathcal{D}_{s}$, these are given in Appendix \ref{Appendix3})
\item If $\overline{f^{\Gamma}}$ is a cube, there are three components lying above $\Gamma$. Otherwise, there is only one component lying above $\Gamma$.
\end{enumerate}

%\begin{rem}
%This technique was explained in \cite[Chapter 8]{tropabelian} to find the Berkovich skeleton of genus three curves. 
%\end{rem}
%<<Opmerking: deze methode werd in Hel17 gebruikt>>

    %This is just a hyperelliptic curve    

\section{Twisting data}\label{TwistingData}

In this section, we will give additional \emph{twisting data} for the degree three abelian covering $\overline{C}\rightarrow{D}$ that completely determines the intersection graph $\Sigma(\overline{\mathcal{C}})$ in terms of $\Sigma(\mathcal{D})$. This twisting data is necessary in the sense that there might be multiple covering graphs for a certain set of covering data. See Figure \ref{83eplaatje} for an example of two different $S_{3}$-coverings with the same covering data. 

\begin{figure}[h!]
\centering
\includegraphics[scale=0.3]{{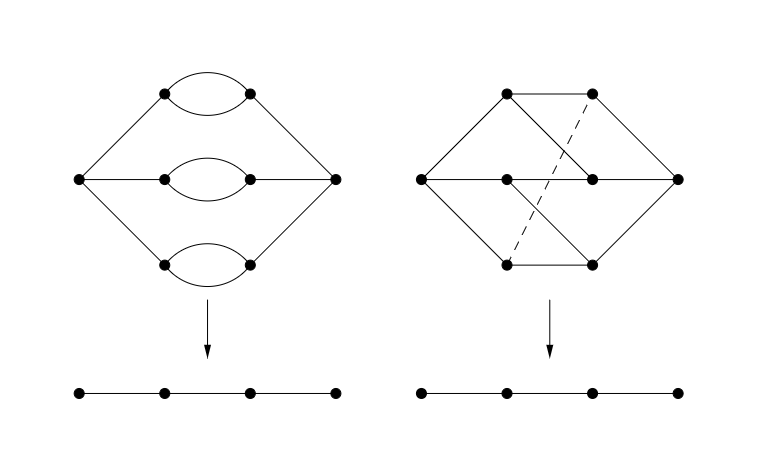}}
\caption{\label{83eplaatje} {\it{Two different $S_{3}$-coverings with the same covering data.}}}
\end{figure}

For abelian covers of the projective line there is no need for this additional data (because every such covering corresponds to a $2$-cocycle on a tree, which is therefore trivial), so this can be seen as a natural first example of the phenomenon. %where this twisting data is needed for coverings of the projective line.

To make everything as explicit as possible, we have added a section containing algebraic representations of the components in $\Sigma(\mathcal{D})$ in Appendix \ref{Appendix3}. Given these representations, the problem is reduced to finding a cube root of the reduced form $\overline{(y-\sqrt{27}q)^{\Gamma}}$ in the function field $k(\Gamma)$ of every component $\Gamma\subset\mathcal{D}_{s}$.

\subsection{Reconstructing $\Sigma(\overline{\mathcal{C}})$}

%This twisting data is necessary%This twisting data is only necessary if there is a connected subgraph of$\Sigma(\mathcal{D})$ with a nonzero Betti number such that the covering is trivial on that subgraph. In that case, there are 
%In this section we will reconstruct the intersection graph of $\overline{\mathcal{C}}$ using the covering data obtained earlier and certain twisting data on the model corresponding to the quadratic subfield $K(D)$.  
%The twisting data will be given as a 2-cocycle on a subgraph of the intersection graph of $\mathcal{D}$. This requires some explicit computations on the components of $\mathcal{D}$, which we will give here.

Let $\overline{C}\rightarrow{\mathcal{D}}\rightarrow{\mathcal{D}_{S}}$ be disjointly branched morphisms associated to the morphisms $\overline{C}\rightarrow{D}\rightarrow{\mathbb{P}^{1}}$, as in Section \ref{QuadraticSubfieldTechnique}. Consider a vertex $v$ in $\Sigma(\mathcal{D})$ with corresponding component $\Gamma\subset{\mathcal{D}_{s}}$. Let $v'$ be any vertex in $\overline{\mathcal{C}}$ lying above $v$ and let $\Gamma'$ be its corresponding component. If $D_{v'/v}=\mathbb{Z}/3\mathbb{Z}$, then there are no options for the edges lying above $v$: they are connected to $v'$. 
%To assign a 2-cocycle to $\Sigma(\mathcal{D})$, 
We now remove these "ramified parts" of $\Sigma(\mathcal{D})$ and then consider the local \'{e}tale equations.
\begin{mydef}
Let $\Sigma(\overline{\mathcal{C}})\rightarrow{\Sigma(\mathcal{D})}$ be the degree three abelian morphism of intersection graphs coming from a disjointly branched morphism $\overline{\phi}$. Let $U(\Sigma(\mathcal{D}))\subset{\Sigma(\mathcal{D})}$ be the (possibly disconnected and incomplete) subgraph of $\Sigma(\mathcal{D})$, consisting of all edges and vertices that have trivial decomposition groups. We call this graph the \emph{unramified} part of $\Sigma(\mathcal{D})$. % for $\phi_{\Sigma}$.
\end{mydef}

Let $v$ be a vertex in $U(\Sigma(\mathcal{D}))$ with corresponding component $\Gamma=\Gamma_{1}$. %We will also write $\Gamma=\Gamma_{1}$. 
The abelian covering is given on the level of function fields as
\begin{equation}
K(D)\rightarrow{K(D)[w]/(w^3-f)},
\end{equation}
where $f:=y-\sqrt{27}q$. We now consider the $\Gamma$-modified form of $f$. That is, we set $k=v_{\Gamma}(f)$ and consider
\begin{equation}
f^{\Gamma}=\dfrac{f}{\pi^{k}},
\end{equation}
so that $v_{\Gamma}(f^{\Gamma})=0$. This means that we can now safely consider the image of $f^{\Gamma}$ in the function field $k(\Gamma)$. We then have
\begin{lemma}\label{FactorizationVertexSplit}
\begin{equation}
\overline{f^{\Gamma}}=h_{v}^{3}
\end{equation}
for some $h_{v}\in{k(\Gamma)}$. In particular, the components lying above $\Gamma$ are given by the prime ideals
\begin{equation}
\mathfrak{q}_{i}=\mathfrak{p}+(w-\zeta^{i}h_{v}),
\end{equation}
where $\zeta$ is a primitive third root of unity and $i\in\{0,1,2\}$.
\end{lemma}  
\begin{proof}
Since there are three vertices lying above $v$, we know that 
\begin{equation}
\text{div}(\overline{f^{\Gamma}})=3D
\end{equation}
for some divisor $D$. Indeed, otherwise the extension $z^3=\overline{f^{\Gamma}}$ would be ramified at some point and thus there would only be one component lying above it. Suppose now that $D$ is not a principal divisor. Then $D$ is a three-torsion point in $\text{Pic}^{0}(\Gamma)$. The corresponding extension would then give a {\it{connected}} unramified covering of $\Gamma$, which contradicts the fact that there are three vertices lying above $v$. This finishes the proof.
\end{proof}

To construct the twisting data, we first gather some standard facts about ordinary double points of length one that are probably familiar to the reader. Let $P\in\mathcal{D}$ be an intersection point and let $A:=\mathcal{O}_{\mathcal{D},P}$. We write $\mathfrak{p}_{1}$ and $\mathfrak{p}_{2}$ for the generic points of the components passing through $P$. Since $P$ is an ordinary double point with length one, we find that
\begin{equation}
\hat{A}=\hat{\mathcal{O}}_{\mathcal{D},P}\simeq{R[[x,y]]/(xy-\pi)}.
\end{equation}
We have the following 
\begin{lemma}\label{LabelOrdDouble}
Let $A$ be as above. Then
\begin{enumerate}
\item $A$ is a unique factorization domain.
\item There exist $x_{1}$ and $y_{1}$ in $A$ such that $v_{\mathfrak{p}_{1}}(x_{1})=0$, $v_{\mathfrak{p}_{2}}(x_{1})=1$, $v_{\mathfrak{p}_{1}}(y_{1})=1$ and $v_{\mathfrak{p}_{2}}(y_{1})=0$ and $x_{1}y_{1}=\pi$.
\item Every element $f\in{K(\mathcal{D})}$ can be written uniquely as
\begin{equation}
f=x^{i}_{1}\cdot{y^{j}_{1}}\cdot{u},
\end{equation}
where $(i,j)\in\mathbb{Z}^{2}$ and $u\in{A}^{*}$. 
\end{enumerate}
\end{lemma}
\begin{proof}
We have that $\hat{A}$ is a unique factorization domain, so by \cite[Lemma 1.2]{samdomains} we find that $A$ is a unique factorization domain\footnote{One could also reason as follows. The ring $\hat{A}$ is regular, so $A$ is regular. A famous result by Auslander and Buchbaum then says that any regular local ring is a unique factorization domain. The lemma cited above is far easier to prove however, only needing Nakayama's Lemma and the Mittag-Leffler condition.}.

%Let $\mathfrak{p}_{1}$ and $\mathfrak{p}_{2}$ be the generic points of the components going through $P$. 
By the approximation theorem for valuations (\cite[Chapter 9, Lemma 1.9]{liu2}), we can find an element $x_{1}\in{K(\mathcal{D})}$ such that $v_{\mathfrak{p}_{1}}(x_{1})=0$ and $v_{\mathfrak{p}_{2}}(x_{1})=1$. Since a normal domain is the intersection of its localizations, we find that $x_{1}\in{A}$. The special fiber of $\mathcal{D}$ is reduced, so we find that $v_{\mathfrak{p}_{i}}(\pi)=1$ for both $i$. Now consider the element $y_{1}:=\dfrac{\pi}{x_{1}}$. We find that $v_{\mathfrak{p}_{1}}(y_{1})=1$ and $v_{\mathfrak{p}_{2}}(y_{1})=0$ and thus $y_{1}\in{A}$. The unique factorization as stated in the Lemma now directly follows. % from our earlier result. %We can now uniquely factorize every element $f\in{K(\mathcal{D})}$ as
%\begin{equation}
%f=x_{1}^{i}y_{1}^{j}u,
%\end{equation}   
%where $u\in{A}*$ and $(i,j)\in\mathbb{Z}^{2}$.
\end{proof}

We now return to our element $f=y-\sqrt{27}q$ and its corresponding normalized form $f^{\Gamma}$, where we we take $\Gamma=\overline{\{\mathfrak{p}_{1}\}}$. As before, we focus on an intersection point $P$ that corresponds to an edge $e\in{}U_{\phi}(\Sigma(\mathcal{D}))$. We have that $v_{\mathfrak{p}_{1}}(f^{\Gamma})=0$, so that we can write
$f^{\Gamma}=x_{1}^{i}{f'}$ for some $f'$. In fact, by the Poincar\'{e}-Lelong formula (See \cite[Corollary 5.1]{tropabelian} or \cite[Theorem 5.15.5]{BPRa1}), we must have $i=3n$. We now consider the element $f'$. We then have $f'\in{A}$. In fact, we see that $v_{\mathfrak{p}_{i}}(f')=0$ for both $i$, so by Lemma \ref{LabelOrdDouble}, we find that $f'\in{A^{*}}$.

%%%%%%%%%%%%%%%%%%%%
%%%%%%%%%%%%%%%%%%%BEGIN OUD
\begin{comment}
We now consider an edge $e\in{}U_{\phi}(\Sigma(\mathcal{D}))$ with corresponding point $P\in\mathcal{D}$ and endpoints $v=v_{1}$ and $v_{2}$, corresponding to $\Gamma=\Gamma_{1}$ and $\Gamma_{2}$ respectively. We assume that $v_{2}\in{}U_{\phi}(\Sigma(\mathcal{D}))$, because otherwise there would be no interesting twisting data. We will explicitly give the normalization of $A=\mathcal{O}_{\mathcal{D},P}$ in the function field $K(\overline(C))$. By the Poincar\'{e}-Lelong formula (reference here), we find that  %The algebra extension
\begin{equation}
\text{ord}_{\overline{P}}(\overline{f})=3n
\end{equation} 
for some $n\in\mathbb{Z}$. We now choose an element $x_{e}\in\mathcal{O}_{\mathcal{D},P}$ such that $\overline{x_{e}}$ is a uniformizer at $\overline{P}$. This implies that $x_{e}$ is a uniformizer for the discrete valuation ring corresponding to $\Gamma_{2}$. Since $P$ is regular and semistable, we find that $y_{e}=\pi/{x_{e}}$ is an element of $\mathcal{O}_{\mathcal{D},P}$ satisfying $x_{e}y_{e}=\pi$. We can now write
\begin{equation}
\overline{f^{\Gamma}}=\overline{x_{e}}^{3n}(g_{e,1})^3
\end{equation}    
for some $g_{e,1}\in{k(\Gamma)}$.% where we assume that $g_{e}$ is not a cube of some function.

We then have
\begin{lemma}
\begin{equation}
f':=\dfrac{f^{\Gamma_{1}}}{x_{e}^{3n}}\in\mathcal{O}^{*}_{\mathcal{D},P}.
\end{equation}
\end{lemma}
\begin{proof}
...
\end{proof}
\end{comment}
%%%%%%%%%%%%%%%%%%%%%%%%%
%%%%%%%%%%%%%%%%%%%%%%EINDE OUD

We now consider the equation on the generic fiber $w^{3}=f$. We assume that the extension is already unramified above $\Gamma$, so that $v_{\mathfrak{p}_{1}}(f)$ is divisible by $3$. We can then normalize this equation to obtain
\begin{equation}
w'^3=f^{\Gamma}.
\end{equation} 

We now consider the element $w''=\dfrac{w'}{x^{n}_{1}}$ in the function field of $K(\mathcal{C})$.

\begin{cor}
The algebra $A[w'']$ is finite \'{e}tale over $A=\mathcal{O}_{\mathcal{D},P}$.%, where $w'=\dfrac{w}{x_{e}^{n}}$.
\end{cor}
\begin{proof}
We have $w''^3=f'$, %\in{A}^{*}$, 
so it is finite. It is standard \'{e}tale by $f'\in{A^{*}}$, and thus also \'{e}tale.
\end{proof}

By Lemma \ref{FactorizationVertexSplit}, we can now factorize $f^{\Gamma}$ and even $f'$ as a cube in $k(\Gamma_{1})$ and $k(\Gamma_{2})$. To avoid confusion, we will write $\text{red}(f',\mathfrak{p}_{i})$ for the image of $f'$ in $\text{Frac}(A/\mathfrak{p}_{i})=k(\mathfrak{p}_{i})$. We then have 
\begin{align*}
\text{red}(f',\mathfrak{p}_{1})&=\overline{g_{1}}^3,\\
\text{red}(f',\mathfrak{p}_{2})&=\overline{g_{2}}^3.
\end{align*}
Note that we can evaluate both $g_{1}$ and $g_{2}$ at the point $P$. 
The components lying above $\Gamma_{1}$ and $\Gamma_{2}$ are now given by the prime ideals
\begin{align*}
\mathfrak{q}_{1,i}&=\mathfrak{p}_{1}+(w''-\zeta^{i}g_{1}),\\
\mathfrak{q}_{2,i}&=\mathfrak{p}_{2}+(w''-\zeta^{i}g_{2}).
\end{align*}
We denote the corresponding components by $\Gamma_{1,i}$ and $\Gamma_{2,i}$ for $i\in\{0,1,2\}$.
Now consider the component $\Gamma_{1,i}$ labeled by $w''=\overline{g_{1}}$. Evaluating $f'$ at $P$, we obtain
\begin{equation}
g_{1}(P)^3=g_{2}(P)^3.
\end{equation}
In other words, there exists a $j\in\{0,1,2\}$ such that $g_{1}(P)=\zeta^{j}g_{2}(P)$. This implies that $\Gamma_{1,0}$ is connected to $\Gamma_{2,i}$ and by the cyclic $\mathbb{Z}/3\mathbb{Z}$-action, this also determines the rest of the edges. We now give a summary of the procedure: %summarize this procedure in the following:

\begin{center}
[{\bf{Procedure for linking components}}]
\end{center}
\begin{enumerate}
\item Consider the component labeled by $w''=g_{1}$ lying above $\Gamma_{1}$.   
\item There exists an $j$ such that 
\begin{equation}
g_{1}(P)=\zeta^{j}\cdot{}g_{2}(P).
\end{equation} 
\item Connect the vertex labeled by $w''=g_{{1}}$ to the vertex labeled by $w=\zeta^{j}g_{2}$.
\item The other vertices and connecting edges lying above $\Gamma_{1}$ are now completely determined by the cyclic $\mathbb{Z}/3\mathbb{Z}$-action. 
\end{enumerate}
 
 By considering these functions $g_{1,P}$ for various intersection points $P$, we now obtain a $2$-cocycle on the intersection graph as follows: we define
 \begin{equation}
 \alpha(e_{1},e_{2})=\dfrac{g_{1,P_{1}}(P_{1})}{g_{1,P_{2}}(P_{2})},
 \end{equation}
 where $P_{i}$ is the intersection point corresponding to the edge $e_{i}$. This can be seen as a generalization of the usual $2$-cocycle one obtains when studying the Picard group of a reduced (possibly reducible) curve with ordinary singularities over a field $k$.

\begin{exa}\label{3Tors}
[{\bf{3-torsion on an elliptic curve}}]
Let us take $p=\alpha\cdot{}x$ and $q=\dfrac{1}{\sqrt{27}}(ax+b)$, where 
\begin{eqnarray*}
a&=&\dfrac{\pi-3}{2},\\
b&=&\dfrac{\pi-1}{2},\\
\alpha^3&=&1/4.
\end{eqnarray*}
The curve we are interested in is then given by
\begin{equation*}
z^3+pz+q=0.
\end{equation*}
Note that this curve has genus $0$, so we already know the minimal skeleton. Nonetheless, we will run the algorithm to show that there are some interesting features.

The $3:1$ covering given by $(z,x)\longmapsto{x}$ has discriminant
\begin{equation*}
\Delta=4p^3+27q^2=x^3+(ax+b)^2.
\end{equation*}
The intermediate curve given by
\begin{equation*}
D:y^2=\Delta
\end{equation*}
has genus 1, with a $3$-torsion point $P=(0,b)$.
%%%%
(\footnote{The way we created this example is as follows. We took the family with $p(x)=x$ and $q(x)=ax+b$ linear and we imposed two conditions: that $x=-1$ be a zero of $\Delta$ and that $\Delta'(-1)=0$ (the derivative with respect to $x$). This then implies that the singular point is different from the $3$-torsion point. This can also be used to create examples of higher genus.}) 
%%%%%
That is, we have
\begin{equation*}
\text{div}(y-(ax+b))=3P-3\cdot\infty.
\end{equation*}
We then easily see that $D$ has split multiplicative reduction and that $P$ doesn't reduce to the singular point. In other words, $P$ defines a $3$-torsion point in the \emph{toric} part of the Jacobian of $D$. 

\begin{figure}[h!]
\centering
\includegraphics[scale=0.27]{{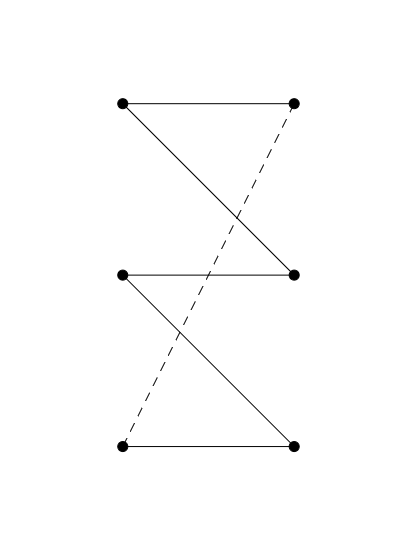}}
\caption{\label{7eplaatje} The intersection graph of the Galois closure in Example \ref{3Tors}.}
\end{figure}

\begin{figure}[h!]
\centering
\includegraphics[scale=0.18]{{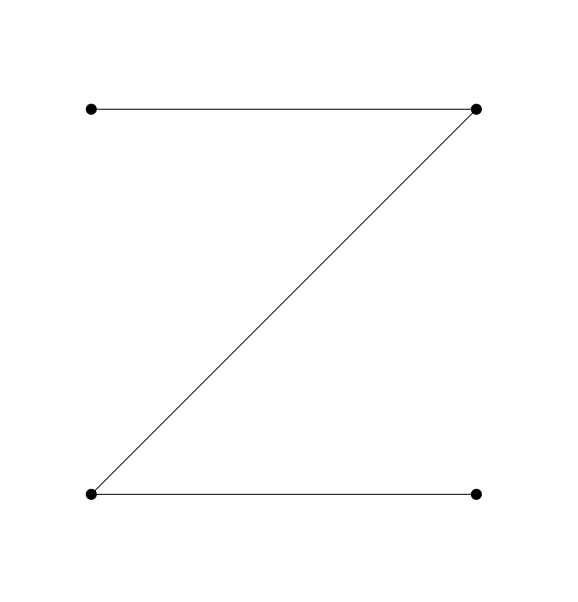}}
\caption{\label{8eplaatje} The intersection graph of the quotient of the Galois closure under a subgroup of order two in Example \ref{3Tors}.}
\end{figure}

At any rate, after a transformation $x\longmapsto{x+1}$ we obtain the equation
\begin{equation*}
y^2=(x-1)^3+(a(x-1)+b)^2=x^3+x^2(a^2-3)+\pi{x}.
\end{equation*}
Let $y'=\dfrac{y}{x}$. 
Taking the model defined by
\begin{equation*}
xt=\pi,
\end{equation*}
we obtain the equation
\begin{equation}\label{EquationExample5}
y'^2=x+a^2-3+t.
\end{equation}
We thus see that we have an intersection graph with two vertices and two edges, giving the multiplicative reduction.  %It is clear that $P$ doesn't reduce to the singular point and thus we obtain a toric $3$-torsion element. 
For $t=0$, we know that there exists a function $g$ such that $\overline{g^3}=\overline{y-\sqrt{27}q}$. We will find this function now.\\
Plugging in $t=0$ in Equation \ref{EquationExample5}, we obtain
\begin{equation*}
y'^2=x-3/4.
\end{equation*}
We thus see that $y'$ parametrizes the corresponding projective line. We write (without the reduction bar for $t=0$):
\begin{equation*}
y-\sqrt{27}q=xy'+3x/2-1=(y'^2+3/4)y'+3/2(y'^2+3/4)-1=(y'+1/2)^3.
\end{equation*}

%\begin{figure}[h!]
%\centering
%\includegraphics[scale=0.4]{{Graph7.png}}
%\caption{\label{7Graph} The covering in Example \ref{3Tors}. }%Two coverings of a graph that have the same covering data but are not isomorphic.}% {\it{The graph of the function $\phi$ considered in Example \ref{SecExa2}.}}}
%\end{figure}
Thus $g=y'+1/2$ solves the corresponding problem. We have 
\begin{eqnarray*}
g(x_{0})&=&-\zeta,\\
g(x_{1})&=&-\zeta^2,
\end{eqnarray*}
where $\zeta$ is a primitive third root of unity and $x_{0}$ and $x_{1}$ are the intersection points. The covering graph is now given by Figure \ref{7eplaatje} and the quotient graph under a subgroup of order two by Figure \ref{8eplaatje}. Note that all components in these figures have genus zero.

%\ref{7Graph}.  
 %that we obtain an element of the identity component of the Jacobian. %The singular point here is given by $(\ov
\end{exa}

\section{The algorithm}

In this section, we assemble the pieces from the previous sections into an algorithm for calculating the Berkovich skeleton of a curve with a degree three covering to the projective line. There are actually two algorithms for the covering data, so the reader can choose whichever method he prefers. The author is under the impression that the method presented in Section \ref{InertTechnique} (using inertia groups) is faster than the one in Section \ref{QuadraticSubfieldTechnique} (using the quadratic subfield), since it doesn't require any Laplacian computations on nontrivial graphs.  % that might have a strictly positive Betti number.  %using the inertia groups of intermediate components is much faster than the one using the quadratic subfield. % although he may be proven wrong 
\begin{comment} 
 Let $C$ be given by
\begin{equation}
z^3+pz+q=0
\end{equation}
for polynomials $p$ and $q$ in $K[x]$. 
\end{comment}
\begin{algo}\label{Algorithm}
{}
\begin{center}
{\it{[The Berkovich skeleton of a curve with a degree three covering to the projective line]}}
\end{center}
\begin{flushleft}
{\bf{Input}}: $p,q\in{K[x]}$.
\end{flushleft}
\begin{enumerate}
\item Let $C$ be given by the equation $z^3+pz+q=0$. If the equation is reducible, then the covering does not have degree three. 
\item If $\Delta=4p^3+27q^2\in{K}$, then $C\rightarrow{\mathbb{P}^{1}}$ is superelliptic over a quadratic extension (namely $K(\sqrt{\Delta})$) of $K$. % and $\overline{C}$ is not geometrically irreducible in this case. 
Use the techniques in \cite{supertrop} to determine the Berkovich skeleton. Otherwise, the Galois closure $\overline{C}$ is geometrically irreducible. It is described by the equation $w^6+2\sqrt{27}qw^3-4p^3=0$.
\item Construct the tropical separating tree $\Sigma(\mathcal{D}_{S})$ for the semistable model $\mathcal{D}_{S}$ as described in Appendix \ref{Appendix2}. Here $S=\text{Supp}(p,q,\Delta)$.
\item Determine the covering data for $\overline{C}\rightarrow{\mathbb{P}^{1}}$ using Section \ref{InertTechnique} or \ref{QuadraticSubfieldTechnique}.
\item Determine the twisting data using Section \ref{TwistingData} and use this to determine the intersection graph of $\overline{\mathcal{C}}$.
\item Calculate the genera of the vertices in $\Sigma(\overline{\mathcal{C}})$ using the Riemann-Hurwitz formula.  %to calculate the genera 
\item Take the quotient of $\Sigma(\overline{\mathcal{C}})$ under the subgroup of order two corresponding to $C$ by Galois theory. The resulting graph is the intersection graph of $\mathcal{C}$.
\item Calculate the genera of the vertices in $\Sigma(\mathcal{C})$ using the Riemann-Hurwitz formula. 
\item Calculate the lengths of the edges in $\Sigma(\mathcal{C})$ using Equation \ref{InertiaFormula}. 
\item Contract any "leaves" to obtain the graph $\Sigma'(\mathcal{C})$.
%\item $\Sigma'(\mathcal{C})$ %There is only one graph $\Sigma(\mathcal{C})$ that satisfies the given covering data and twisting data. Determine this graph.
\end{enumerate}
\begin{flushleft}
{\bf{Output}}: The Berkovich skeleton $\Sigma'(\mathcal{C})$.
\end{flushleft}
\end{algo}
\begin{proof}
({\it{Correctness of the algorithm}})
For the covering data and the twisting data, we refer the reader to Sections \ref{CoveringData} and \ref{TwistingData}. The fact that the quotient graph is equal to the intersection graph of the quotient is \cite[Lemma 7, page 16]{tropabelian}. Contracting any leaves then automatically yields the Berkovich skeleton.  %The noncontractible part of the intersection graph then yields the minimal Berkovich skeleton, which is exactly the output.  %Since the obtained graph is the intersection graph of a semistable model, it contains the minimal Berkovich skeleton. %we find by \cite[]{BPRa1}%The intersection graph of this semistable model, minus the edges and vertices that are contractible, is now isometrically isomorphic to the Berkovich skeleton of $C$ by <Baker/Rabinoff>. This finishes the proof.   %We note that determining the unique graph that satisfies that covering and twisting data is a \emph{finite} computation. 
\end{proof}

\begin{rem}
Let us point out a particular application of the algorithm. Let $C$ be any curve of genus three. If $C$ is hyperelliptic, then the Berkovich skeleton can be quite easily found, see \cite[Algorithm 4.2]{supertrop}. If $C$ is not hyperelliptic, then the canonical embedding embeds $C$ as a quartic in $\mathbb{P}^{2}$ and after a finite extension, $C$ contains a rational point. Projecting down onto this point then gives a $3:1$ cover $\phi$ onto the projective line. This means that $\phi$ is generically given by $(x,z)\mapsto{x}$, with $z^3+pz+q=0$ for some $p$ and $q$. We can then use the algorithm to find the Berkovich skeleton of $C$. %In other words, this algorithm gives the Berkovich skeleton of \emph{any} genus three curve.
\end{rem}

%In this section, we give the covering data %for the morphism of intersection $\Sigma(\mathcal{C})\longrightarrow{\mathcal{D}}$ using the quadratic

%We then easily obtain the following
%\begin{lma}
%The $2$-cocycle $\alpha$ is well-defined up to elements 
%\end{lma}

 %eWe can now explicitly find the normalization of the covering above any edge $e$ corresponding to an intersection point $x\in\mathcal{D}$.  %We then find that $h_{\Gamma} 

 %, we have 

%If $D_{v'/v}=\mathbb{Z}/3\mathbb{Z}$, then no additional twisting data for its neighboring edges is needed. 

% We first note that the intersection graph of $\mathcal{D}$ is completely determined by the covering data for the degree two morphism $D\rightarrow{\mathbb{P}^{1}}$. 

%\subsection{The 2-cocycle}

\section{Semistability of elliptic curves using a degree three covering}

As an application of the methods presented in the previous sections, we reprove the criterion:
\begin{equation}
"v(j)<0 \text{ if and only if }E\text{ has split multiplicative reduction over an extension of }K".
\end{equation}
Here $j$ is the $j$-invariant of $E$. Note that $E$ having multiplicative reduction over an extension of $K$ is equivalent to $E$ admitting a semistable model $\mathcal{E}$ over $R'$ such that the intersection graph has Betti number one. We will use the latter in this section. 
%<<Referentie/uitleg toevoegen>>
%Here, having multiplicative reduction is equated to having a semistable model $\mathcal{E}$ over $R$ with intersection graph having Betti number one. 

Let us take an elliptic curve $E/K$. Over an extension $K'$ of $K$, one can then find an equation of the form
\begin{equation}\label{EquationElliptic}
x^3+Ax+B+y^2=0
\end{equation}
%(the unusual notation will be explained shortly)
 for some $A$ and $B$ in $K'$. Just as in \cite{Silv1}, we can assume that $v(A),v(B)\geq{0}$. In fact, we can assume that the equation has been scaled such that either $v(A)=0$ or $v(B)=0$, which again often requires a finite extension. We will assume that all these extensions have been made and the resulting field will be denoted by $K$. % We will also assume that $K'=K$. 
To prove semistability of the curve, one usually considers the $2:1$ covering given by
\begin{equation*}
\phi(x,y)=x
\end{equation*}
and then uses the branch points to explicitly create the semistable model. We will make life hard for us now and consider a different covering:
\begin{equation*}
\phi(x,y)=y.
\end{equation*}
This gives a degree three morphism $E\longrightarrow{\mathbb{P}^{1}}$ with corresponding extension of function fields $K(y)\subset{K(E)}$. We will use the quadratic subfield of the Galois closure and the corresponding technique described in Section \ref{QuadraticSubfieldTechnique}. The twisting data studied in Section \ref{TwistingData} will not be needed, as we will see that the covering data obtained here determines the covering graph uniquely. %We will see that there is no twisting data needed  

We note that the curve in Equation \ref{EquationElliptic} is in our normal form with $p=A$ and $q=B+y^2$. For psychological reasons, the author chose to revert the minus sign coming from the usual Weierstrass equation (given by $x^3+Ax+B-y^2=0$) to a plus sign. 

Consider the $K(A,B)[y]$-algebra $K(A,B)[y][x]/(x^3+Ax+B+y^2)$. We first calculate the discriminant of this algebra. It is given by
\begin{equation*}
\Delta=4A^3+27(B+y^2)^2.
\end{equation*}
We would like to determine whether this is a square or not. To that end, we calculate the discriminant of
\begin{equation*}
\Delta'(y_{1})=4A^3+27(B+y_{1})^2
\end{equation*}
and see that 
\begin{equation*}
\Delta(\Delta'(y_{1}))=(2\cdot{27}\cdot{B})^2-4\cdot{(27)}\cdot{(27B^2+4A^3)}=-(4\cdot{27})^2\cdot{A}^3.
\end{equation*}
Here $y_{1}=y^2$. 
We therefore see that the discriminant $\Delta$ is a square if and only if either $A=0$ or $y=0$ is a zero of $\Delta$. In the latter case we see directly that we must have $4A^3+27B^2=0$, which contradicts the assumption that $E$ is nonsingular. The case $A=0$ is a separate case, where one can easily see that $E$ has potential good reduction. %Note that in this case, the covering $(x,y)\mapsto{y}$ from Equation \ref{EquationElliptic} is visibly Galois. %, as was to be expected.

So let us assume that $A\neq{0}$. Then the discriminant is not a square and we obtain a bonafide extension of degree two given by
\begin{equation*}
z^2=4A^3+27(B+y^2)^2.
\end{equation*}
This is again a curve of genus $1$, which we denote by $E'$. We would like to know the reduction type of this curve. We will do this in terms of the discriminant $\Delta(E)=4A^3+27B^2$. Note that the $E$'s equation has been scaled such that either $v(A)=0$ or $v(B)=0$. %This will generally require a finite extension of $K$.\\

We now consider the following possible scenarios for $A,B$ and $\Delta(E)$:
\begin{enumerate}\label{ScenariosEllipticCurve1}
\item $v(A)=v(B)=0$ and $v(\Delta(E))>0$.
\item $v(A)=0$, $v(B)\geq{0}$ and $v(\Delta(E))=0$.
\item $v(A)>0$, $v(B)=0$ and $v(\Delta(E))=0$.
%\item $v(A)=0$, $v(B)\geq{0}$ and $v(\Delta(E))=0$.
\end{enumerate}
\begin{lemma}\label{Onecase1}
Every elliptic curve $E/K$ belongs to exactly one of the three cases described above.
\end{lemma}
\begin{proof}
Let $E$ be given as in Equation \ref{EquationElliptic}. If $v(A)>0$, then by assumption we must have $v(B)=0$ and thus $v(\Delta)=0$. This means that we are in Case 3. Suppose that $v(A)=0$. If $v(B)>0$, then $v(\Delta)=0$ and we are in Case 2. If $v(B)=0$, then there are two possibilities: either $v(\Delta)>0$ or $v(\Delta)=0$. These are cases 1 and 2 respectively. It is now clear from the nature of these cases that they are mutually exclusive. This finishes the proof.  
\end{proof}

%One can quite easily see that these are mutually exclusive.

\begin{theorem}\label{ScenariosEllipticCurve2}
Let $\overline{E}$ be the Galois closure of the morphism $\phi$. %Consider a triple $(E,E',\overline{E})$ with corresponding intersection graphs 
%\begin{equation*}((\mathcal{G}(E),\mathcal{G}(E'),\mathcal{G}(\overline{E})).
%\end{equation*}
%Then every type of triple $(v(A),v(B),v(\Delta(E)))$ (as described above) corresponds to exactly one triple of intersection graphs. The correspondence is given by
For every type of $\{v(A),v(B),v(\Delta)\}$ as described above, there exists a disjointly branched morphism  $\overline{\mathcal{E}}\rightarrow{\mathcal{D}_{\mathbb{P}^{1}}}$ giving the following intersection graphs:  %such that morphism of intersection graphs $\Sigma(\overline{\mathcal{E}})\rightarrow{\Sigma(\mathcal{D}_{\mathbb{P}^{1}})}$ is given by: %These three cases correspond%be a disjointly branched morphism with corresponding morphism of intersection graphs
\begin{enumerate}
\item Suppose that $v(A)=v(B)=0$ and $v(\Delta(E))>0$. Then $E'$ has multiplicative reduction with intersection graph $\Sigma(\mathcal{E}')$ consisting of two components intersecting in two points. $\Sigma(\overline{\mathcal{E}})$ consists of 3 copies of $\Sigma(\mathcal{E}')$ meeting in one vertex and $E$ has multiplicative reduction. The corresponding intersection graph $\Sigma(\mathcal{E})$ consists of three vertices, connected as in Figure \ref{101eplaatje}.
\item Suppose that $v(A)=0$, $v(B)\geq{0}$ and $v(\Delta(E))=0$. Then all curves involved are nonsingular and the corresponding models have the trivial intersection graph. %The corresponding intersection graphs can be found in Figure \ref{20eplaatje}. 
\item Suppose that $v(A)>0$, $v(B)=0$ and $v(\Delta(E))=0$. Then $E'$ has multiplicative reduction with intersection graph $\Sigma(\mathcal{E}')$ consisting of two components intersecting in two points.  $\Sigma(\overline{\mathcal{E}})$ consists of two elliptic curves meeting twice. $E$ has good reduction, with intersection graph $\Sigma(\mathcal{E})$ as described in Figure \ref{92eplaatje}. 
%\item Suppose that $v(A)=0$, $v(B)\geq{0}$ and $v(\Delta(E))=0$. Then all curves involved are nonsingular.
\end{enumerate}
%The converse for the above also holds.
\end{theorem}
\begin{proof}
We subdivide the proof into three parts, according to the cases given in the statement of the proposition.
\begin{enumerate}
\item %We first determine the reduction type of $E'$. 
We write 
\begin{equation*}
z^2=4A^3+27B^2+2\cdot{27}By^2+27y^4=\Delta(E)+2\cdot{27}By^2+27y^4=\Delta.
\end{equation*}
We label the roots of $\Delta$ by $\alpha_{i}$ for $i\in\{1,2,3,4\}$. 
Since $v(A)=v(B)=0$ and $v(\Delta(E))>0$, we find that two of the roots of $\Delta$ coincide. Let them be $\alpha_{1}$ and $\alpha_{2}$. A Newton polygon computation then shows that  $v(\alpha_{1})=v(\alpha_{2})=v(\Delta)/2$. We therefore construct a tropical separating tree $\Sigma(\mathcal{D}_{S})$ with two vertices and one edge, which has length $v(\Delta)/2$. The reduction graph of $E'$ is then as shown in Figure \ref{101eplaatje} (which contains some spoilers regarding the final product). Indeed, the Laplacian of $\Delta(E)+2\cdot{27}By^2+27y^4$ has slope $\pm{2}$ on $e$, so we obtain two edges. Furthermore, there is only one vertex lying above each vertex of $\Sigma(\mathcal{D}_{S})$, since the roots of $\Delta$ are branch points on these components. We label these components by $\Gamma_{1}$ and $\Gamma_{2}$.  %We then see that $E'$ has multiplicative reduction.

We now consider the degree three covering $\overline{E}\rightarrow{E'}$. We'll use the formulas in Proposition \ref{DivisorDegree3}. Let $f=z-\sqrt{27}q$. We then easily see that
%The extension is given by the function
%\begin{equation*}
%f=z-\sqrt{27}q=z-\sqrt{27}(y^2+B).
%\end{equation*}
%We first have to find the divisor of this function on the generic fiber. This is given by
\begin{equation*}
\text{div}_{\eta}(f)=2\cdot({\infty_{1}})-2\cdot({\infty_{2}}),
\end{equation*}
where the $\infty_{i}$ are the two points at infinity. These points both reduce to smooth points on $\Gamma_{1}$. The corresponding Laplacian is thus trivial on $\Sigma(E')$. Using Proposition \ref{QuadraticSubfield}, we see that there are three edges lying above each of the two in $\Sigma(E')$.  

We now turn to the vertices. The reduced divisor of $f$ on $\Gamma_{2}$ is trivial, so there are three vertices lying above it. For $\Gamma_{1}$, the covering is ramified and thus there is only vertex lying above it. By the Riemann-Hurwitz formula, the covering vertex has genus zero. We thus directly see that the intersection graph must be as in Figure \ref{101eplaatje}. The corresponding quotient and the rest of the lattice is depicted there as well. We see that the Betti number of the quotient is one, implying that $E$ has multiplicative reduction. 

Note that the length of the cycle in $\Sigma(\mathcal{E})$ is the same as the length of the cycle in $\Sigma(\mathcal{E}')$ by inspecting the corresponding inertia groups. From the construction of the tropical separating tree, we then find that the cycle in $\Sigma(\mathcal{E}')$ has length $v(\Delta)/2+v(\Delta)/2=v(\Delta)=-v(j)$ and thus the cycle in $\Sigma(\mathcal{E})$ has the same length. This is another well-known feature of elliptic curves with split multiplicative reduction.  %This cycle has length $v(\Delta)=-v(j)$

\begin{figure}[h!]
\centering
\includegraphics[scale=0.3]{{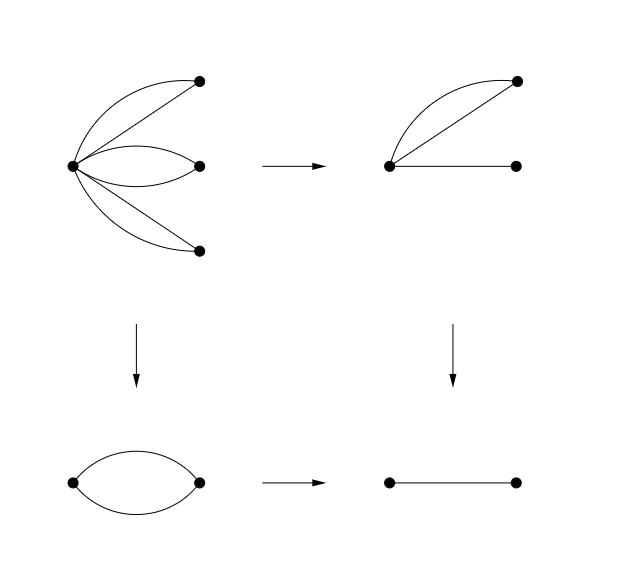}}
\caption{\label{101eplaatje} The Galois closure of graphs in Case I.}%{\it{The Laplacian function $\phi$ of $f$, as in Example \ref{Exa3torsgen2}. The $e_{i}$ denote the three edges between the two vertices.}}} %The dual intersection graph for every reduction type as in Theorem \ref{ThmRedType}.}
\end{figure}
\item A quick calculation shows that all curves in sight are nonsingular. We thus obtain trivial graphs with weights $1,3,1$. Hence $E$ has good reduction. %The corresponding Galois lattice can be found in Figure \ref{20eplaatje}.

\begin{comment}
%%%%%%%%%%PLAATJE
\begin{figure}[h!]
\centering
\includegraphics[scale=0.4]{{Graph20.png}}
\caption{\label{20eplaatje} The Galois closure of graphs in Case II.}%{\it{The Laplacian function $\phi$ of $f$, as in Example \ref{Exa3torsgen2}. The $e_{i}$ denote the three edges between the two vertices.}}} %The dual intersection graph for every reduction type as in Theorem \ref{ThmRedType}.}
\end{figure}
\end{comment}
%%%%%%%%%%%%%%%%%

\item Suppose that $v(A)>0$, $v(B)=0$ and $v(\Delta(E))=0$. ({\footnote{The author has to confess that this feels like we're using too much machinery, because we already know that $E$ has good reduction from the fact that the reduced discriminant is nonzero. Nonetheless, calculating the entire Galois closure shows some interesting features.}})

\begin{figure}[h!]
\centering
\includegraphics[scale=0.3]{{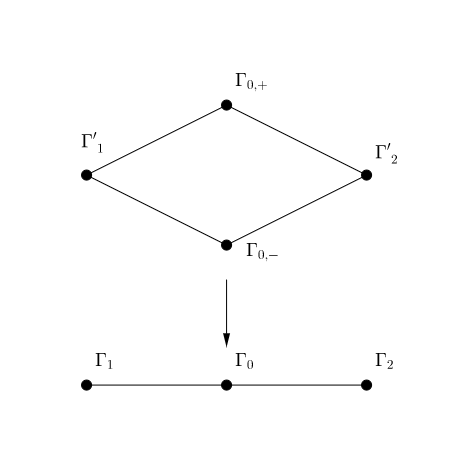}}
\caption{\label{91eplaatje} The hyperelliptic covering in Case III.}%{\it{The Laplacian function $\phi$ of $f$, as in Example \ref{Exa3torsgen2}. The $e_{i}$ denote the three edges between the two vertices.}}} %The dual intersection graph for every reduction type as in Theorem \ref{ThmRedType}.}
\end{figure}

We label the roots of $\Delta$ by $\alpha_{i}$ for $i\in\{1,2,3,4\}$. We quickly find that the roots reduce to roots of the equation $y^2+B=0$. We have two roots (say $\alpha_{1}$ and $\alpha_{2}$) reducing to $y=\sqrt{-B}$ and two other roots ($\alpha_{3}$ and $\alpha_{4}$) reducing to $y=-\sqrt{-B}$. A Newton polygon calculation then shows that $v(\alpha_{1}-\alpha_{2})=3v(A)/2$ and $v(\alpha_{3}-\alpha_{4})=3v(A)/2$. We therefore construct a tropical separating tree with three vertices $\Gamma_{0}$, $\Gamma_{1}$ and $\Gamma_{2}$ as in Figure \ref{91eplaatje}. This figure also contains the corresponding degree two covering of intersection graphs. Note that the edges $\Gamma_{1}\Gamma_{0}$ and $\Gamma_{0}\Gamma_{2}$ both have length $3v(A)/2$. Since the morphism of intersection graphs is \'{e}tale above these edges, we find that the edges lying above them also have length $3v(A)/2$.

\begin{figure}[h!]
\centering
\includegraphics[scale=0.3]{{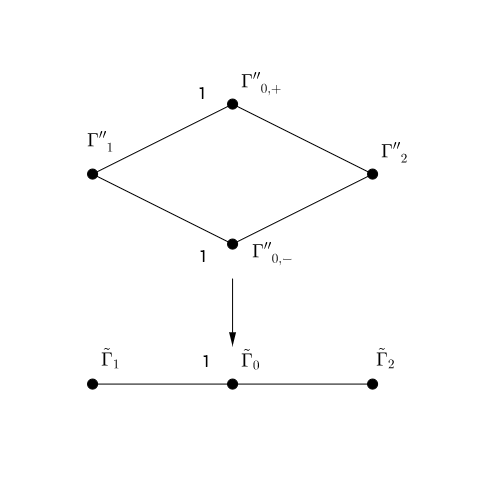}}
\caption{\label{92eplaatje} The Galois closure and its quotient in Case III. }%{\it{The Laplacian function $\phi$ of $f$, as in Example \ref{Exa3torsgen2}. The $e_{i}$ denote the three edges between the two vertices.}}} %The dual intersection graph for every reduction type as in Theorem \ref{ThmRedType}.}
\end{figure}

  %Note that $\infty$ reduces to $\Gamma_{0}$. % and that the two points at infinity red We then find that the corresponding degree two covering of intersection graphs is given by Figure ... 

Since $\Sigma(\mathcal{D})$ has Betti number one, we see that $E'$ has multiplicative reduction. Note that there are two points at infinity $\infty_{i}$ and that they reduce to different components: $\infty_{1}\mapsto{\Gamma_{0,+}}$ and $\infty_{2}\mapsto{\Gamma_{0,-}}$. Let $f=y-\sqrt{27}q$. We now find that $\text{div}_{\eta}(f)=2(\infty_{1})-2(\infty_{2})$. The reduction of the divisor of $f$ is then given by % takes the form
\begin{equation*}
\rho(\text{div}_{\eta}(f))=2(\Gamma_{1})-2(\Gamma_{2}).
\end{equation*}
We then find that the Laplacian corresponding to $f$ has slope $\pm{1}$ on every edge. It is not divisible by three, so there is only edge lying above every edge, again by Proposition \ref{QuadraticSubfield}. The Galois closure can now be found in Figure \ref{92eplaatje}. The components $\Gamma''_{0,i}$ both have genus one. Indeed, the morphisms $\Gamma''_{0,i}\rightarrow{\Gamma'_{0,i}}$ are ramified above the edges and $\infty_{i}$, giving a total of three ramification points per component. The Riemann-Hurwitz formula then gives the desired genus. We now see that the quotient consists of three vertices, where the middle vertex has genus $1$. In other words, $E$ has good reduction. %is as in Figure \ref{92eplaatje}  %The Riemann-Hurwitz formula then gives the  ramify in the edges % We now see that $\Gamma''_{1}\rightarrow\Gamma_{1}$ and $\Gamma'_{2}\rightarrow{\Gamma_{2}}$ are both ramified in three points. The two points corresponding to the edges ramify in both morphisms and the points $\infty_{1}$ and $\infty_{2}$ ramify in $\Gamma_{1}$ and $\Gamma_{2}$ respectively. We thus see that the intersection graph for $\overline{E}$ is as in Figure \ref{19eplaatje}. 

\end{enumerate}
\end{proof}

\begin{cor}
An elliptic curve $E$ has split multiplicative reduction over an extension of $K$ if and only if $v(j)<0$.
\end{cor}
\begin{proof}
Every elliptic curve can be put in exactly one of the three scenarios considered in Proposition \ref{ScenariosEllipticCurve2} by Lemma \ref{Onecase1}. %with corresponding values of $A$, $B$ and $\Delta$. % of \ref{ScenariosEllipticCurve1}. 
Case 1 corresponds exactly to $v(j)<0$ by the calculation $v(j)=v(1728\cdot\dfrac{4A^3}{\Delta})=-v(\Delta)<0$. Cases 2 and 3 then naturally correspond to $v(j)\geq{0}$, giving the Corollary.  %Suppose that $E$ has multiplicative reduction. Then we see from \ref{{ScenariosEllipticCurve1}} that $v(j)=v(1728\cdot\dfrac{4A^3}{\Delta})=-v(\Delta)<0$, as desired.\\
%Suppose that $v(j)<0$. Then   %Using proposition \ref{ScenariosEllipticCurve2} we then see that only the first scenario gives elliptic curves with multiplicative reduction. This happens exactly when $v(j)=v(1728\cdot\dfrac{4A^3}{\Delta})=-v(\Delta)<0$, as desired. 
\end{proof}

\begin{appendix}

\section{Explicit normalizations for tame $S_{3}$-coverings of discrete valuation rings}\label{Appendix1}

In this section, %we will give the equations that describe the Galois closure of a degree three field extension with Galois group $S_{3}$. %a short review of the %short recapitulation on the "Cardano formul 
we will give the proofs for Propositions \ref{InertS3} and \ref{DivisorDegree3}. % The proofs  largely consist of calculations involving normalizations. 
We first give a short review of the equations defining the Galois closure of a degree three separable field extension with Galois group $S_{3}$, after which we turn to normalizing discrete valuation rings in these extensions. %In doing this, we will also find the va

The set-up is as follows. Let $R$ be a discrete valuation ring with quotient field $K$, residue field $k$, uniformizer $\pi$ and valuation $v$. As already mentioned in Section \ref{TameDisc}, the residue field $k$ is not assumed to be algebraically closed, since we want to use these results for discrete valuation rings coming from irreducible components in a semistable model $\mathcal{C}$. %the residue fields of components $\Gamma_{ in a semistable model % since we'll be using the results here for valuations corresponding to components in the special fiber of a semistable model. 
We will denote the maximal ideal $(\pi)$ in $R$ by $\mathfrak{p}$. We will assume that the characteristic of $K$ is zero %$\text{char}(K)=0$ 
and that the characteristic of the residue field is coprime to six. Furthermore, we will assume that $K$ contains a primitive third root % By Hensel's Lemma, we then find that $K$ contains a primitive root 
of unity $\zeta_{3}$ and a primitive fourth root of unity $\zeta_{4}$. %Note that this implies that $\sqrt{27}=3\sqrt{3}\in{K}$. %  We will also assume that $K$ contains a primitive third root of unity $\zeta_{3}$. 
The third root of unity allows us to use Kummer theory for abelian coverings of degree $3$. The fourth root of unity allows us to change the sign in the discriminant of a cubic equation, see the equations below. Note that it also implies that $\sqrt{3}\in{K}$ and thus $\sqrt{27}\in{K}$. 

 % and $2$. To make the upcoming formulas somewhat simpler, we will also assume that $\sqrt{27}\in{K}$. This is not essential to the core of the material and can be circumvented.

\subsection{The Galois closure of an irreducible degree three extension}%The Cardano formulas for a cubic equation}

%The title of this section can be a bit misleading, since we're not going to give the full Cardano formulas. Instead, we will derive the equations that define the Galois closure of a non-Galois degree three extension $L\supset{K}$, which is the implicit backbone of the entire technique. % This will suffice for our purposes.  %(namely, for finding normalizations) %which is not Galois.  %for the Galois closure of % and show that they 

Let $K\subset{L}$ be a field extension of degree $3$. Let $z$ be any element in $L\backslash{K}$. After a translation, its minimal polynomial in $K[x]$ is given by $f(x):=x^3+p\cdot{}x+q$, leading to the equation  %given by an equation of the form
\begin{equation}\label{MainEq2}
z^3+p\cdot{z}+q=0,
\end{equation}
where $p,q\in{K}$. Its discriminant is then given by
\begin{equation}
\Delta_{f}:=-(4p^3+27q^2).
\end{equation} 
We first have the following
\begin{lemma}
Let $\overline{L}$ be the Galois closure of $L/K$. Then $\text{Gal}(\overline{L}/K)=S_{3}$ if and only if the discriminant $\Delta_{f}$ is not a square in $K$.
\end{lemma} 
\begin{proof}
See \cite[Proposition 22.4.]{stewartgalois}. 
\end{proof}%such that the Galois group of $\overline{L}/K$ is $S_{3}$ (which is equivalent to the discriminant not being a square in $K$). 
We now assume that the Galois group of $\overline{L}/K$ is $S_{3}$. We have the following chain of subgroups
\begin{equation}
(1)\vartriangleleft{H}\vartriangleleft{S_{3}},
\end{equation} 
where $H$ has order $3$ and index $2$. In other words, $S_{3}$ is solvable. Using this fact, the famous Cardano formulas then express the roots of $f(x)$ in terms of radicals. % two consecutive radical extensions.
 Let us quickly summarize the procedure. One considers the following equation: %takes a quadratic extension of $L$ given by
\begin{equation}\label{EquationW}
w^2-3wz-p=0.
\end{equation}  
If it has a root in $L$, we take that root and call it $w$. Otherwise, we take a quadratic extension to obtain the desired $w$. It will later turn out that this is exactly the extension to the Galois closure. Note that $w\neq{0}$. Indeed, otherwise we would have $p=0$ and this would imply that $L/K$ is abelian, a contradiction. One can also assume that $w\neq{0}$ in the abelian case, because at least one of the roots of Equation \ref{EquationW} has to be nonzero. 

  At any rate, this $w$ then satisfies the (probably more familiar) equation
\begin{equation}
z=w-\dfrac{p}{3w}.
\end{equation}
This is also known as {\it{Vieta's substitution}}. Plugging this into Equation \ref{MainEq2} then quickly leads to to a sextic equation
\begin{equation}\label{EquationFull}
w^6+qw^3-\dfrac{p^3}{27}=0,
\end{equation}
which is quadratic in $w^3$. We can now consider the element
\begin{equation}
\tau(w):=\dfrac{p}{3w}.
\end{equation}
%We note here that if $w$ satisfies Equation \ref{EquationFull}, then 
The reader can immediately check that if $w$ satisfies Equation \ref{EquationFull}, then $\tau(w)$ also satisfies the same equation. % First of all, note that this is well-defined: if $w$ were zero, then $p$ would be zero, which would imply that $L/K$ is abelian, a contradiction.
 We now have
\begin{lemma}
$\tau(w)\neq{w}$.
\end{lemma}
\begin{proof}
Suppose that $\tau(w)=w$. Then $w^2=\dfrac{p}{3}$ and $w^6=\dfrac{p^3}{27}$. Substituting this into Equation \ref{EquationFull}, we find
\begin{equation}
\dfrac{p^3}{27}+qw^3-\dfrac{p^3}{27}=qw^3=0.
\end{equation}
In other words, either $q=0$ or $w=0$. If $q=0$, then Equation \ref{MainEq2} is reducible, a contradiction. We already saw that $w=0$ is impossible, so we arrive at the desired result. %If $w=0$, then $p=0$. This implies that Equation \ref{MainEq2} is abelian, contradicting the fact that its Galois group is $S_{3}$. This finishes the proof.
\end{proof}
%Using the fact that Equation
Completing the square in Equation \ref{EquationFull}, we now obtain: % is quadratic in $w^3$, we can now write
\begin{equation}
(w^3+q/2)^2=\dfrac{p^3}{27}+\dfrac{q^2}{4}=\Delta':=\dfrac{\Delta}{4\cdot{27}},
\end{equation}
where $\Delta=4p^3+27q^2$. We thus see that the quadratic subfield $K(y):=K[y]/(y^2-\Delta')$ is contained in $K(w,z)=K(w)$ (where $w$ is a root of Equation \ref{EquationW}). Note that $K(y)$ is indeed a field, since $\Delta'=\dfrac{\Delta}{4\cdot{27}}$ is minus the discriminant, which is not a square in $K$ by assumption on the Galois group. Using the fact that field degrees are multiplicative, we then see that $K(w)\supset{K(y)}$ has degree three. This then implies that $K(w)$ has degree six over $K$, which also gives the irreducibility of Equation \ref{EquationFull}.
We now have
\begin{lemma}
The field extension $K(w)\supset{K(z)}\supset{K}$ is Galois of order six. As such, it is the Galois closure of $K(z)/K$. The two automorphisms given by
\begin{align*}
\sigma(w)&=\zeta_{3}\cdot{w},\\
\tau(w)&=\dfrac{p}{3w}
\end{align*}    
generate the Galois group. 
Here $\sigma$ has order three and $\tau$ has order two.
\end{lemma}
\begin{proof}
By basic field theory, $\tau$ defines an automorphism of order two on $K(w)$. One then also easily finds that $\sigma$ is an automorphism of order three. Note that they both fix the underlying field $K$. We now have two automorphisms that generate a group $<\sigma,\tau>=:H\subset{\text{Aut}(K(w))}$ with order equal to the degree of the field extension (namely six). This implies that $K(w)/K$ is Galois with Galois group $H$, as desired. 
\end{proof}
%We can in fact easily see that $K(w)/K$ is Galois of order six: the automorphisms 
%have order three and two respectively and fix $K$. 

Let us now perform some cosmetic changes that remove the fractions from the equations. We scale Equation \ref{EquationFull} slightly using the variable change
\begin{equation}
w'=\dfrac{w}{c\sqrt{3}},
\end{equation}
where $c^3=2$. Writing $w$ for $w'$,  
this then leads to the equation
\begin{equation}
w^6+2\sqrt{27}qw^3-4p^3=0.
\end{equation} 
Completing the square and taking the extension $K\subset{K[y]/(y^2-\Delta)}$ with $\Delta=4p^3+27q^2$, we find that
%Completing the square, we then find
%\begin{equation}\label{IntermediateExtension}
%(w^3+\sqrt{27}q)^2=\Delta,
%\end{equation}
%where $\Delta=4p^3+27q^2$. 
%In other words, we have found a quadratic subextension
%\begin{equation}
%K\subset{K[y]/(y^2-\Delta)}\subset{K(w,z)}=K(w).
%\end{equation}
%The extension $K(y):=K[y]/(y^2-\Delta)\subset{K(w)}$ has degree three by the identity $[K(z):K]\cdot{[K(w):K(z)]}=[K(w):K(y)][K(y):K]$ and the fact that Equation \ref{MainEq2} describes an irreducible equation over $K$.  %multiplicativity of field degrees and 
%From Equation \ref{IntermediateExtension} we find the equation
\begin{equation}
w^3=\pm{}y-\sqrt{27}q.
\end{equation}
Throughout the paper, we in fact take the extension
\begin{equation}
w^3=y-\sqrt{27}q.
\end{equation}
The other extension (namely $w^3=-y-\sqrt{27}q$) is just the extension corresponding to $\tau(w)$, where $\tau$ is now given by $\tau(w)=\dfrac{cp}{w}$.
\begin{cor}\label{GalClos1}
The Galois closure of $L\supset{K}$ is given by the two extensions
\begin{equation}
K\subset{K(y)}\subset{K(w)},
\end{equation}
where
\begin{equation}
w^3=y-\sqrt{27}q
\end{equation}
and
\begin{equation}
y^2=\Delta.
\end{equation}
\end{cor}

%From Equation \ref{IntermediateExtension}, we then find the abelian extension
%\begin{equation}
%K[y]/(y^2-\Delta)\subset{K[y]/(y^2-\Delta)[w]/(w^3-(y-\sqrt{27}q))}.
%\end{equation}%%
%Note that it must have degree three, since $L$
%It can then be checked that $\sigma(w)=\zeta_{3}w$ and $\tau(w)=\dfrac{cp}{w}$ are automorphisms of the field extension $K(w)/K$ of order $3$ and $2$ respectively. This implies that $K(w)/K$ is Galois and thus $K(w)/K$ is the Galois closure of $K(z)/K$. %The extension is hence Galois
%where $K(w)$ is obtained by adjoining a root of Equation \ref{EquationFull}. 
%We then see that we have an extension from this field (note that it is a field by assumption on the discriminant) to $K

  %By scaling, we obtain a monic equation of the same form with $p,q\in{R}$. We now consider the $K$-algebra
%\begin{equation}
%A:=K[y]/(y^2-\Delta),
%\end{equation}
%where $\Delta=4p^3+27q^2$. Note that $y^2-\Delta$ is irreducible, since the discriminant is a non-square by assumption.   

\subsection{Normalizations}\label{Normalizations}

Now let $R\subset{K}$ be as in the beginning of the Appendix. Let $L\supset{K}$ be a degree three field extension and $z\in{L}\backslash{K}$. After a translation, $z$ satisfies
\begin{equation}
z^3+p\cdot{z}+q=0,
\end{equation}
where $p$ and $q$ are in $K$. By scaling, we can even assume that $p,q\in{R}$. 
Let $K'\subset{\overline{L}}$ be the quadratic subfield, $A$ the integral closure of $R$ in $K'$ and $B$ the integral closure of $R$ in $\overline{L}$. Let $\mathfrak{q}$ be any prime in $B$ lying above $\mathfrak{p}=(\pi)$. We will give explicit equations for the ring $B$ and use those to give formulas for $|I_{\mathfrak{q}}|$, the inertia group of $\mathfrak{q}$.

We consider three cases:

\begin{enumerate}
\item $3v(p)>2v(q)$,
\item $3v(p)<2v(q)$,
\item $3v(p)=2v(q)$.
\end{enumerate}

In every case, we start with a computation of the integral closure $A$ and then deduce $B$ from $A$. From Corollary \ref{GalClos1}, we see that the extension $K\subset{K'}$ is given by
\begin{equation}
y^2=\Delta=4p^3+27q^2.
\end{equation}

\subsubsection{Case I}
Suppose that $3v(p)>2v(q)$. We let $p=\pi^{k_{1}}u_{1}$ and $q=\pi^{k_{2}}u_{2}$ for units $u_{i}$. We then find the integral equation
\begin{equation}\label{EquationCase1}
(\dfrac{y}{\pi^{k_{2}}})^2=4\pi^{3k_{1}-2k_{2}}u_{1}^3+27u_{2}^2.
\end{equation}
Let $y'=\dfrac{y}{\pi^{k_{2}}}$.
Reducing Equation \ref{EquationCase1} modulo $\mathfrak{p}$ yields the equation
\begin{equation*}
\overline{(y')^2}=\overline{27u_{2}^2}.
\end{equation*}
Or in other words: $\overline{y'}=\pm{\sqrt{27}u_{2}}$. In other words: $A$ is completely split over $R$. The primes are then given by:
\begin{eqnarray}
\mathfrak{q}_{1}&=&(\pi,\sqrt{27}u_{2})\\
\mathfrak{q}_{2}&=&(\pi,-\sqrt{27}u_{2}).
\end{eqnarray}
Note that this implies that $\pi$ is again a uniformizer of $A_{\mathfrak{q}_{i}}$ for both $i$.  %prime $\mathfrak{q}$ lying above $\mathfrak{p}$. 
We then have the following Lemma:
\begin{lemma}\label{caseone}
Let $3v(p)>{2v(q)}$. Then
\begin{eqnarray*}
v_{\mathfrak{q}_{1}}(y-\sqrt{27}q)&=&3v(p)-v(q)\\
v_{\mathfrak{q}_{1}}(y+\sqrt{27}q)&=&v(q).\\
v_{\mathfrak{q}_{2}}(y+\sqrt{27}q)&=&3v(p)-v(q)\\
v_{\mathfrak{q}_{2}}(y-\sqrt{27}q)&=&v(q).
\end{eqnarray*}
\end{lemma}
\begin{proof}
We write $y-\sqrt{27}q=\pi^{k_{2}}\cdot{(y'-\sqrt{27}u_{2})}$ and use the relation%Since 
\begin{equation}
(y'-\sqrt{27}u_{2})(y'+\sqrt{27}u_{2})=4\pi^{3k_{1}-2k_{2}}u_{1}^3.
\end{equation}
Note that $y'-\sqrt{27}u_{2}$ and $y'+\sqrt{27}u_{2}$ are coprime, so that $y'+\sqrt{27}u_{2}$ is invertible at $\mathfrak{q}_{1}$.  We then see that the desired valuation is given by
\begin{equation}
v_{\mathfrak{q}_{1}}(y-\sqrt{27}q)=v_{\mathfrak{q}_{1}}(\pi^{k_{2}})+v_{\mathfrak{q}_{1}}(4\pi^{3k_{1}-2k_{2}}u_{1}^3)=k_{2}+3{k_{1}}-2k_{2}=3{k_{1}}-k_{2}.
\end{equation}
Using again that $y'+\sqrt{27}u_{2}$ is invertible at $\mathfrak{q}_{1}$, we obtain that $v_{\mathfrak{q}_{1}}(y+\sqrt{27}q)=v(\pi^{k_{2}})=k_{2}$, as desired.

The other two cases for $\mathfrak{q}_{2}$ follow in a completely analogous way and are left to the reader.  
\end{proof}

We can now give the order of the inertia groups for primes $\mathfrak{q}$ in $B$. 
\begin{lemma}
Let $3v(p)>2v(q)$. Then
\begin{enumerate}
\item $|I_{\mathfrak{q}}|=3$ $\iff$ $3\nmid{v(q)}$.
\item  $|I_{\mathfrak{q}}|=1$ $\iff$ $3| v(q)$. 
\end{enumerate}
\end{lemma}
\begin{proof}
This follows from Lemma \ref{caseone}.
\end{proof}

\subsubsection{Case II}
Suppose that $3v(p)<2v(q)$ and let $p=\pi^{k_{1}}u_{1}$ and $q=\pi^{k_{2}}u_{2}$ for units $u_{i}$, as before. We subdivide this case into two subcases:
\begin{enumerate}
\item [{\bf{(A)}}] $v(p)$ is divisible by $2$,
\item [{\bf{(B)}}] $v(p)$ is not divisible by $2$.
\end{enumerate}

%\subsubsection{Case II.A}
Suppose that $v(p)$ is divisible by $2$. We start with the equation $(y-\sqrt{27}q)(y+\sqrt{27}q)=4p^3$. Dividing by $\pi^{3k_{1}}$, we obtain
\begin{equation}
(\dfrac{y-\sqrt{27}q}{\pi^{3k_{1}/2}})(\dfrac{y+\sqrt{27}q}{\pi^{3k_{1}/2}})=4u_{1}^3.
\end{equation}
We thus see that $\dfrac{y-\sqrt{27}q}{\pi^{3k_{1}/2}}$ and $\dfrac{y+\sqrt{27}q}{\pi^{3k_{1}/2}}$ are invertible. Note that the reduced equation is
\begin{equation}
\overline{y'^2}=\overline{4u_{1}^3}, 
\end{equation}
which might be reducible or irreducible, depending on whether $u_{1}$ is a square in the residue field $k$. In either case, we have the following
\begin{lemma}\label{Case2A}
Let $3v(p)<2v(q)$ and suppose that $v(p)$ is divisible by $2$. Then $|I_{\mathfrak{q}}|=1$. %Furthermore, $|D_{\mathfrak{q}}|=2$ precisely when $p$ is not a square in $k$.
\end{lemma} 

This concludes the determination of the inertia groups for the first case. We would now also like to give the valuation of $y-\sqrt{27}q$ at a prime in $\text{Spec}(A)$ lying above $\mathfrak{p}$. As noted above, there are two cases to consider: the case where $\mathfrak{p}$ is split in $A$ and the case where $\mathfrak{p}$ is not split in $A$. We saw that being split in $A$ is equivalent to $\overline{u}_{1}$ being a square in $k$.

Let us consider the case where $\mathfrak{p}$ is split in $A$. We can then write $\overline{u_{1}}=h^2$, where $\overline{h}\in{k}$. Let $h$ be a lift of $h$ to $A$. Then there are two primes lying above $\mathfrak{p}$:
\begin{align*}
\mathfrak{q}_{1}&=(y'-2h^3,\pi)\\
\mathfrak{q}_{2}&=(y'+2h^3,\pi).
\end{align*}
We can now give $v_{\mathfrak{q}_{i}}(y\pm\sqrt{27}q)$:
\begin{lemma}
Suppose that $\mathfrak{p}$ is split in $A$. Let $\mathfrak{q}_{i}$ be the primes in $\text{Spec}(A)$ lying above $\mathfrak{p}$. Then
\begin{eqnarray*}
v_{\mathfrak{q}_{i}}(y\pm\sqrt{27}q)&=&3v_{}(p)/2.
%v_{\mathfrak{q}_{i}}(y+\sqrt{27}q)&=&3v_{\alpha}(p)/2.
\end{eqnarray*}
\end{lemma}
\begin{proof}
Using 
\begin{equation}\label{EquationPDivisible}
(\dfrac{y-\sqrt{27}q}{\pi^{3k_{1}/2}})(\dfrac{y+\sqrt{27}q}{\pi^{3k_{1}/2}})=4u_{1}^3,
\end{equation}
we see that $\dfrac{y-\sqrt{27}q}{\pi^{3k_{1}/2}})$ and $(\dfrac{y+\sqrt{27}q}{\pi^{3k_{1}/2}})$ are invertible and thus $v_{\mathfrak{q}_{i}}(\dfrac{y\pm\sqrt{27}q}{\pi^{3k_{1}/2}})=0$. %Indeed, they are invertible by this relation. 
Since $R\rightarrow{A_{\mathfrak{q}_{i}}}$ is \'{e}tale for both $i$, we obtain that $\pi$ is again a uniformizer. This quickly gives the lemma.
\end{proof}

Suppose now that $\mathfrak{p}$ is not split in $A$. There is one prime lying above $\mathfrak{p}$, namely
\begin{equation}
\mathfrak{q}=(y'^2-4u_{1}^3,\pi).
\end{equation}
We then have
\begin{lemma}
Let $\mathfrak{q}$ be the only prime lying above $\mathfrak{p}$. Then
\begin{equation}
v_{\mathfrak{q}}(y\pm\sqrt{27}q)=3v(p)/2.
\end{equation}
\end{lemma}
\begin{proof}
Using Equation \ref{EquationPDivisible} again, we see that $v_{\mathfrak{q}}(\dfrac{y\pm\sqrt{27}q}{\pi^{3k_{1}/2}})=0$. Since $R\rightarrow{A_{\mathfrak{q}}}$ is \'{e}tale, we have that $\pi$ is again a uniformizer and the result follows. 
\end{proof}

This concludes the case where $v(p)$ is divisible by $2$.
%\subsubsection{Case II.B}
Now suppose that $v(p)$ is not divisible by $2$. We claim that $|I_{\mathfrak{q}}|=2$. We write $3v(p)=2k+1$ and find the equation
\begin{equation}\label{EquationRamifiedP}
(\dfrac{y-\sqrt{27}q}{\pi^{k}})(\dfrac{y+\sqrt{27}q}{\pi^{k}})=4\pi\cdot{}u_{1}^3.
\end{equation}
Writing $y'=\dfrac{y}{\pi^{k}}$ and $q'=\dfrac{q}{\pi^{k}}$, we see that there is only one prime lying above $\mathfrak{p}$ in this algebra, namely
\begin{equation}
\mathfrak{q}'=(y'-\sqrt{27}q', \pi)=(y'-\sqrt{27}q').
\end{equation}
Note that $\mathfrak{q}'$ is principal and thus $A$ is normal. The fiber over $\pi$ is of the form $\overline{y'^2}=\overline{0}$, showing that the extension is ramified.
Since the inertia group is cyclic inside $S_{3}$ and we already know that $|I_{\mathfrak{q}}|$ is greater than or equal to $2$, we find that $|I_{\mathfrak{q}}|=2$. We summarize this in a lemma:
\begin{lemma}\label{Case2B}
Let $3v(p)<2v(q)$ and suppose that $v(p)$ is not divisible by $2$. Then $|I_{\mathfrak{q}}|=2$.
\end{lemma} %Since $v(q)-k>0$, we see that the extension is ramified of degree $2$, with $\pi$ not a uniformizer anymore. In fact, $(\dfrac{y-\sqrt{27}q}{\pi^{k}})$ is now a uniformizer. Since the inertia group for a tame covering of discrete valuation rings is abelian, we now know that $|I_{\mathfrak{q}}|=2$ for any prime $\mathfrak{q}$ lying above $\mathfrak{p}$. 
%In both cases, we let $\mathfrak{q}$ be a prime lying above $\mathfrak{p}$.  %In the first case, there are two primes lying above $\mathfrak{p}$

Note now that there are only two options for $v(p)$: it is either divisible by two or it is not. Using this observation, we then also obtain the reverse statement of Lemmas \ref{Case2A} and \ref{Case2B}. We could also obtain this from the following lemma:
\begin{lemma}
Let $\mathfrak{q}'$ be the only prime lying above $\mathfrak{p}\in\text{Spec}(A)$. Then
\begin{eqnarray*}
v_{\mathfrak{q}'}(y\pm\sqrt{27}q)&=&3v_{}(p).
%v_{\mathfrak{q}}(y+\sqrt{27}q)&=&3v_{\mathfrak{p}}(p).
\end{eqnarray*}
\end{lemma}
\begin{proof}
This follows from Equation \ref{EquationRamifiedP}, noting that $v_{\mathfrak{q}'}(\pi)=2$ and $v_{\mathfrak{q}'}(y\pm\sqrt{27}q)=1$.
\end{proof}

\subsubsection{Case III}
Suppose that $3v(p)=2v(q)$. Let $\Delta:=4p^3+27q^2$.  We again consider two cases:
\begin{enumerate}
\item [{\bf{(A)}}] $v(\Delta)$ is divisible by $2$,
\item [{\bf{(B)}}] $v(\Delta)$ is not divisible by $2$.
\end{enumerate}

Suppose first that $v(\Delta)$ is divisible by $2$. We then see that the extension is unramified in the quadratic subfield.
\begin{lemma}
Suppose that $3v(p)=2v(q)$ and that $v(\Delta)$ is divisible by $2$. Then $|I_{\mathfrak{q}}|=1$.
\end{lemma}
\begin{proof}
 We consider the equation
\begin{equation*}
(y'-\dfrac{\sqrt{27}q}{\pi^{v(q)}})(y'+\dfrac{\sqrt{27}q}{\pi^{v(q)}})=\dfrac{4p^3}{\pi^{2v(q)}}.
\end{equation*}
Note that the righthand side is invertible, implying that the elements on the lefthand side are also invertible. % this again implies that 
%with $r_{1}=$ and $r_{2}=$. 
From $v(\Delta)\equiv{0}\mod{2}$, we obtain that $\pi$ is again a uniformizer at the two points lying above it. We denote them by $\mathfrak{q}_{i}$. %_{\pm}$. 
This gives
\begin{lemma}
 %Suppose that $3v(p)=2v(q)$, $v(\Delta)\equiv{0}\mod{2}$ and
 Let the $\mathfrak{q}_{i}$ be the two primes lying above $\mathfrak{p}$. % and
Then
\begin{eqnarray*}
v_{\mathfrak{q}_{i}}(y\pm\sqrt{27}q)&=&v(q).
%v_{Q_{\pm}}(y+\sqrt{27}q)&=&v_{\alpha}(q(x)).
%v_{Q_{-}}(y+\sqrt{27}q)&=&v_{\alpha}(q(x)).\\
%v_{Q_{+}}(y+\sqrt{27}q)&=&v_{\alpha}(q(x))\\
%v_{Q_{+}}(y-\sqrt{27}q)&=&v_{\alpha}(q(x)).
\end{eqnarray*}
\end{lemma} 
\begin{proof}
This follows as before, noting that $(y'-\dfrac{\sqrt{27}q}{\pi^{v(q)}})$ and $(y'+\dfrac{\sqrt{27}q}{\pi^{v(q)}})$ are invertible and that $v_{\mathfrak{q}_{i}}(\pi)=1$.
\end{proof}
This then also quickly gives the rest of the lemma: since $3|3v(p)$, we find that $3|2v(q)$, implying that $3|v(q)=v_{\mathfrak{q}_{i}}(y\pm\sqrt{27}q)$. As always, this implies that the abelian extension $K'\subset{\overline{L}}$ is unramified above the $\mathfrak{q}_{i}$. %This gives that the extension above the $\mathfrak{q}_{i}$ is unramified.  % since these valuations determine the ramification. 
\end{proof}
Now for the second case with $v(\Delta)\equiv{1}\mod{2}$. We find that there is only one point (denoted by $\mathfrak{q}$) that lies above $\mathfrak{p}$. The valuation of $\pi$ in $\mathfrak{q}$ is then $2$. This then gives
\begin{lemma}
Suppose that $3v(p)=2v(q)$, $v(\Delta)\equiv{1}\mod{2}$ and let $\mathfrak{q}$ be the only prime lying above $\mathfrak{p}$. 
\begin{equation}
v_{\mathfrak{q}}(y\pm\sqrt{27}q)=2v(q)=3v(p).
%v_{Q}(y+\sqrt{27}q)&=&2v_{\alpha}(q(x))=3v_{\alpha}(p(x)).
%v_{Q_{-}}(y+\sqrt{27}q)&=&v_{\alpha}(q(x)).\\
%v_{Q_{+}}(y+\sqrt{27}q)&=&v_{\alpha}(q(x))\\
%v_{Q_{+}}(y-\sqrt{27}q)&=&v_{\alpha}(q(x)).
\end{equation}
\end{lemma}
\begin{proof}
This follows exactly as before: we have the equation
\begin{equation}
(y'-\dfrac{\sqrt{27}q}{\pi^{v(q)}})(y'+\dfrac{\sqrt{27}q}{\pi^{v(q)}})=\dfrac{4p^3}{\pi^{2v(q)}},
\end{equation}
implying that the lefthand side is invertible. Since $v_{\mathfrak{q}}(\pi)=2$, we obtain the lemma by a simple calculation.
\end{proof}

\section{Tropical separating trees and semistable models}\label{Appendix2}

In this section we will associate a tree to a set of elements of $\tilde{K}:=K\cup{\{\infty\}}$. % the roots of a set of polynomials $\{f_{1},...,f_{r}\}\in{K}[x]$
% a tree. 
This tree canonically gives a semistable model $\mathcal{D}$ of $\mathbb{P}^{1}$ such that the closure of these elements consists of disjoint, smooth sections. We will first give the construction in terms of blow-ups and then give the construction in terms of $\pi$-adic expansions. %This section is mostly written from an algorithmic/motivational point of view, %and as such we invite the reader to fill in the details for parts of the proofs.   
 \subsection{Construction in terms of blow-ups}
Let 
\begin{equation*}
S:=\{\alpha_{1},...,\alpha_{r}\}\subset{\tilde{K}^{r}}.
\end{equation*}
We will assume that no $\alpha_{i}$ is equal to another $\alpha_{j}$. %We will inductively construct a tree.  %The $\alpha_{i}$ 
Consider the standard model $\mathbb{P}^{1}_{R}$ given by glueing the rings $R[x]$ and $R[1/x]$. Let $\tilde{x}$ be any closed point of the special fiber $\mathbb{P}^{1}_{k}$.
 Let
\begin{equation*}
S_{\tilde{x}}:=\{\alpha\in{S}:r(\alpha)=\tilde{x}\},
\end{equation*}
where $r(\cdot)$ is the reduction map associated to $\mathbb{P}^{1}_{R}$.
This partitions the original set $S$ (because points have a unique reduction point by the fact that our ring $R$ is Henselian).\\
We label the points in $\mathbb{P}^{1}_{k}$ that the set $S$ reduces to by
\begin{equation*}
x_{1},x_{2},...,x_{l}.
\end{equation*}
We will then write $S_{i}$ for $S_{x_{i}}$. Let us consider the blow-up of $\mathbb{P}^{1}_{R}$ at $x_{i}$. We will later give an interpretation using $\pi$-adic expansions. We denote the blow-up by $P_{x_{i}}$. Let $z$ be any closed point in the exceptional divisor of $P_{x_{i}}$. Consider the set
\begin{equation*}
r_{P_{x_{i}}}^{-1}(z)=\{\alpha\in{S_{i}}:r_{P_{x_{i}}}(\alpha)=z\}.
\end{equation*}
By varying $z$ over the closed points of the exceptional divisor of $P_{x_{i}}$, we obtain a partition of $S_{i}$. We label the points in the exceptional divisor that the set $S_{i}$ reduces to by
\begin{equation*}
x_{i,1},x_{i,2},...,x_{i,l_{i}}.
\end{equation*}
As before, we then define $S_{i_{1},i_{2}}=r_{P_{x_{i_{1}}}}^{-1}(x_{i_{1},i_{2}})$. Continuing this process, we then obtain sequences of sets
\begin{equation*}
S_{i_{1},i_{2},...,i_{t}}.
\end{equation*}

\begin{lemma}\label{LemmaLargeT}
For $t$ large, we have $|S_{i_{1},i_{2},...,i_{t}}|=1$.
\end{lemma}
\begin{proof}
The easiest way to see this is using $\pi$-adic expansions. We defer this to the next section. The basic idea is that the semistable model corresponding to $S_{i_{1},i_{2},...,i_{t}}$ separates certain $\pi$-adic expansions up to a certain height.%The points agree up to a certain height, the first set after that separates the points.$\longrightarrow{}$ uitwerken en uitleggen in termen van expansies.  
\end{proof}

We can now define the {\bf{tropical separating tree}} of $S$.
\begin{mydef}\label{TropSep}
Consider the set of all $S_{i_{1},i_{2},...,i_{t}}$ as constructed above. We consider the following inclusions:
\begin{equation*}
S_{i_{1},i_{2},..,i_{t-1},i_{t}}\subset{S_{i_{1},i_{2},...,i_{t-1}}}.
\end{equation*}
Consider the graph $\Sigma_{S}$ consisting of all  $S_{i_{1},i_{2},...,i_{t}}$ as vertices. The edge set consists of all pairs of vertices such that
\begin{equation*}
S_{i_{1},i_{2},..,i_{t-1},i_{t}}\subset{S_{i_{1},i_{2},...,i_{t-1}}}.
\end{equation*}
Furthermore, we create one vertex $S_{\emptyset}$ that is connected to all $S_{i}$.
The {\bf{tropical separating tree}} for $S$ is the finite complete subgraph consisting of the following vertex set:
\begin{itemize}
\item All vertices $S_{i_{1},i_{2},...,i_{t}}$ such that $|S_{i_{1},i_{2},..,i_{t}}|>1$. %(Note the change in index here).
\item The vertex $S_{\emptyset}$.
\end{itemize}
%It is denoted by $\Sigma_{S}$.
\end{mydef}

\begin{rem}
From Lemma \ref{LemmaLargeT}, we see that the tropical separating tree is indeed finite. 
\end{rem}
\begin{mydef}
The semistable model $\mathcal{D}_{S}$ for $\mathbb{P}^{1}$ constructed before Definition \ref{TropSep} will be referred to as the {\it{separating semistable model}} for $S$. Its intersection graph is the same as $\Sigma_{S}$.
\end{mydef}
%The tropical separating tree comes with a semistable model $\mathcal{D}_{S}$ of $\mathbb{P}^{1}_{K}$: simply consider the semistable model obtained by iterating the blow-ups as mentioned before Lemma \ref{LemmaLargeT}. The intersection graph of this semistable model is then in fact the same as the tropical separating tree. 
We will give some explicit equations for parts of this semistable model $\mathcal{D}_{S}$ in Appendix \ref{Appendix3}. 
\begin{exa}\label{Exa1}
Let us consider the 4 elements
\begin{equation*}
S=\{0,\pi,\pi+\pi^2, \pi+2\pi^2\}.
\end{equation*}
We see that they all reduce to the point $0$. We therefore consider the blow-up, given by
\begin{equation*}
t'\pi=x.
\end{equation*}
The corresponding $t'$-coordinates for the last three points are
\begin{equation*}
S'=\{1,{1+\pi},{1+2\pi}\}.
\end{equation*}
These reduce to the same point given by $(x,t'-1,\pi)$. We consider one patch of the blow-up in this point, given by
%The corresponding $t$-coordinates for the last three points are $1, 1+\pi, 1+2\pi$. These all reduce to the same point given by $(x,t-1,\pi)$. We consider one patch of the blow-up, given by 
\begin{equation*}
R[x,t'][t'']/(t'\pi-x,t''\pi-(t'-1)).
\end{equation*} 
The last two points of $S'$ now have $t''$-coordinates given by
\begin{equation*}
S''=\{1,2\}.
\end{equation*}
We see that these points reduce to different points, meaning we have reached our endpoint.\\
The corresponding tropical separating tree is now given as follows: we have a tree consisting of three vertices. The first set is
\begin{equation*}
S_{0}=\{0,\pi,\pi+\pi^{2},\pi+2\pi^{2}\}.
\end{equation*}
On the blow-up, these reduce to two different points $x_{0,0}$ and $x_{0,1}$. We have
\begin{eqnarray*}
S_{0,0}&=&\{0\}, \\
S_{0,1}&=&\{\pi,\pi+\pi^{2},\pi+2\pi^{2}\}.
\end{eqnarray*}
The graph now has vertex set
\begin{equation*}
V_{\Sigma}=\{S_{\emptyset},S_{0},S_{0,1}\},
\end{equation*}
with edges between $S_{\emptyset}$ and $S_{0}$, and $S_{0}$ and $S_{0,1}$.
%The last three points of $S_{0,1}$ all reduce to different points on the blow-up,
%First example of a separating tree
\end{exa}
%\begin{exa}
%Second example of a separating tree
%\end{exa}
\subsection{Construction using $\pi$-adic expansions}
Let $N\subset{R}$ be a set of representatives of $R/\mathfrak{m}$. We will assume that $0$ represents $\overline{0}$. Then every element $x$ of $R$ can be written {\it{uniquely}} as
\begin{equation*}
x=\sum_{i=0}^{\infty}a_{k}\pi^{k},
\end{equation*}
where $a_{k}\in{N}$. Suppose that we are given a finite set $S$ of elements in $R$. Write every element $\alpha_{i}$ as
\begin{equation*}
\alpha_{i}=\sum_{k=0}^{\infty}a_{i,k}\pi^{k}
\end{equation*}
as above. Furthermore, suppose that every point reduces to the point $\tilde{0}$, that is: $a_{i,0}=0$. We then have
\begin{equation*}
\alpha_{i}=\pi\cdot{f_{i,1}},
\end{equation*}
where 
\begin{equation*}
f_{i,1}={\sum_{k=0}^{\infty}a_{i,k+1}\pi^{k}.}
\end{equation*}%every element is of the form
%\begin{equation*}
%\alpha_{i}=\pi{f_{i}}
%\end{equation*}
%for some $f_{i}$ in $R$. 
As indicated in the previous section, we create the separating tree for $S$ by considering blow-ups in the points where $S$ is not separated. The critical point here is of course $P=(x,\pi)$. Let us consider a specific affine patch of the blow-up of $\text{Spec}{(R[x])}$ in $P$:
\begin{equation*}
R_{1}:=R[x,t']/(t'\pi-x).
\end{equation*}
We then easily see that the $t'$ coordinates of the elements of $S$ are given by
\begin{equation*}
t'(\alpha_{i})=f_{i,1}.
\end{equation*}
Considering the prime ideal $(x-\alpha_{i},t'-f_{i,1},\pi)$, we see that it gives a reduced coordinate $\overline{t'}=\overline{f_{i,1}}=\overline{a_{i,1}}$. \\ 
We state this observation separately:
\begin{itemize}
\item {\it{The extra coordinate $t'$ on the blow-up keeps track of the coefficient $a_{i,1}$ in the $\pi$-adic expansion}}.
\end{itemize}
That is, we have separated these coordinates up to their first ($k=1$) $\pi$-adic coefficient.
Now if, for instance, $a_{1,1}$ and $a_{2,1}$ are the same, we cannot distinguish between them on this blow-up. We therefore blow-up $\text{Spec}(R_{1})$ in the point $Q:=(x-\alpha_{i},t'-a_{1,1},\pi)$. This gives a new algebra $R_{2}:=R_{1}[t'']/(t''\pi-(t'-a_{1,1}))$. We can then write
\begin{eqnarray*}
f_{1,1}-a_{1,1}&=&\pi{f_{1,2}},\\
f_{2,1}-a_{1,1}&=&\pi{f_{2,2}},
\end{eqnarray*}
where
\begin{eqnarray*}
f_{1,2}&=&\sum_{k=0}^{\infty}a_{1,k+2}\pi^{k},\\
f_{2,2}&=&\sum_{k=0}^{\infty}a_{2,k+2}\pi^{k}.
\end{eqnarray*}
This means that the $t''$-coordinates of $\alpha_{1}$ and $\alpha_{2}$ are given by
\begin{eqnarray*}
t''(\alpha_{1})&=&f_{1,2},\\
t''(\alpha_{2})&=&f_{2,2}.
\end{eqnarray*}
As before, we have that their reduced coordinates are now respectively $\overline{t''}=\overline{f_{1,2}}=\overline{a_{1,2}}$ and $\overline{t''}=\overline{f_{2,2}}=\overline{a_{2,2}}$. These new coordinates can be the same of course and we can then continue the blow-up process. At some point however we must have that their $\pi$-adic coefficients are different (at least, if $\alpha_{1}\neq{\alpha_{2}}$). This happens exactly at the $k$-th $\pi$-adic coefficient, where $k=v(\alpha_{1}-\alpha_{2})$. Note that this also proves Lemma \ref{LemmaLargeT}, since in any finite set of distinct elements in $\mathbb{P}^{1}(K)$, elements agree only up to a certain finite height in their $\pi$-adic expansions. 
\begin{rem}
An important observation now is the following: every component $\Gamma$ in the semistable model $\mathcal{D}_{S}$ corresponds to a finite $\pi$-adic expansion
\begin{equation}
z_{\Gamma}=a_{0}+a_{1}\pi^{1}+...+a_{k}\pi^{k}.
\end{equation}
Points $z$ in $\mathbb{P}^{1}(K)$ that have this expansion up to height $k$ will reduce to this component $\Gamma$ if there are no further components $\Gamma'$ (with their own $\pi$-adic expansions) that agree with $z$ up to a higher power of $\pi$. %This then actually gives an easy  interpretation of the Berkovich space associated to $\mathbb{P}^{1}$: it's a bookkeeping mechanism for keeping track of all the $\pi$-adic expansions of elements in $\mathbb{P}^{1}(K)$ and likewise for extensions of $K$.
\end{rem}
\subsubsection{An algorithm for separation}\label{AlgSep1}
Let us give this $\pi$-adic separation process in an algorithmic fashion. We suppose that we are given $n$ points $S:=\{\alpha_{i}\}$ that all reduce to finite points in $\mathbb{P}^{1}_{k}$. That is, $v(\alpha_{i})\geq{0}$. The infinite case is similar. 
\begin{mydef}
We say that $S$ is {\bf{separated up to height }}$k$ if the images of the $\alpha_{i}$ in the ring $R/(\pi)^{k}$ are all different.
\end{mydef}
For any finite set $S$, there exists a finite integer $k$ such that $S$ is separated up to height $k$. For instance, the roots of a single polynomial $f$ are separated up to $v(\Delta(f))$ (which is of course not always the smallest integer that has this property). At any rate, we are now ready to give the separating tree in terms of $\pi$-adic expansions.\newline
{\center{
[{\bf{Algorithm for separating trees}}]
\begin{itemize}
\item Calculate for every $i$ the $\pi$-adic expansion $\alpha_{i}=u_{0,i}+u_{1,i}\pi+u_{2,i}\pi^2+...+u_{k,i}\pi^{k}+r_{i}$, where $r_{i}$ has valuation strictly greater than $k$.
\item First partition: partition $S$ into subsets $S_{j}$ that have the same zeroth order approximation $u_{0,i}$.
\item Second partition: partition every $S_{j_{1}}$ into subsets $S_{j_{1},j_{2}}$ that have the same first order approximation $u_{1,i}$.
\item Third partition: partition every $S_{j_{1},j_{2}}$ into subsets $S_{j_{1},j_{2},j_{3}}$ that have the same second order approximation $u_{2,i}$. 
\item (Iterate the partition process up to $k$).
\item  Construct the finite tropical separating tree $\Sigma_{S}$ according to Definition \ref{TropSep}. %  where every $S_{j}$ corresponds to a distinct element   
\end{itemize}}}

\section{Explicit representations of the components in the quadratic subfield}\label{Appendix3}

Let $\phi: D\rightarrow{\mathbb{P}^{1}}$ be a hyperelliptic covering, given generically by an equation of the form
\begin{equation}
y^2=f(x),
\end{equation}
where we assume that $f(x)$ is a squarefree polynomial. 
Let $S$ be a set in $\mathbb{P}^{1}(K)$ containing the branch locus of $\phi$. The previous section demonstrated a canonical semistable model $\mathcal{D}_{S}$ of $\mathbb{P}^{1}$ that separates $S$ in the special fiber. After a finite extension of $K$, we find that the normalization $\mathcal{D}$ of $\mathcal{D}_{S}$ in $K(D)$ gives a semistable model of $D$ over $R'$. In this section, we give an explicit representation of the residue field extension
\begin{equation}
k(\Gamma)\rightarrow{k(\Gamma')},
\end{equation} 
where $\Gamma\subset{\mathcal{D}_{S,s}}$ is an irreducible component in the special fiber. This representation is needed in the algorithm for the twisting data.

We apply the construction of Appendix \ref{Appendix2} to $S$ and find the separating tree $\Sigma(\mathcal{D}_{S})$. % to the set of roots of $f(x)$ (and $\infty$ if the degree of $f(x)$ is odd). 
A component $\Gamma$ in this tropical separating tree now corresponds to a finite $\pi$-adic expansion
\begin{equation*}
z_{\Gamma}=a_{0}+a_{1}\pi^{1}+...+a_{k}\pi^{k}.
\end{equation*}
We now wish to obtain an expression of $x$ in terms of a local uniformizer (which is $\pi$) and a generator of the residue field of $\Gamma$. This is in fact not too hard: we consider the following chain of blow-ups
\begin{eqnarray*}
R_{0}&=&R[x], \\
R_{1}&=&R_{0}[t_{1}]/(t_{1}\pi-(x-a_{0})), \\
R_{2}&=&R_{1}[t_{2}]/(t_{2}\pi-(t_{1}-a_{1})), \\
\vdots&{}&\vdots\\
R_{k}&=&R_{k-1}[t_{k}]/(t_{k}\pi-(t_{k-1}-a_{k-1})).
\end{eqnarray*}
%The corresponding chain of equalities
This then expresses $x$ in terms of $t_{k}$: $x=g(t_{k})$. We then find $f(x)=f(g(t_{k}))$. To obtain the normalization, we take the highest power of $\pi$ out:
\begin{equation*}
f(x)=f(g(t_{k}))=\pi^{r}h(t_{k}).
\end{equation*}
Note that this $h(t_{k})$ is indeed a polynomial in $t_{k}$, as can easily be seen from the above equations. The normalization is then given by
\begin{equation}\label{NormEq1}
y'^2=h(t_{k}),
\end{equation}
where $y'=\dfrac{y}{\pi^{r/2}}$.
Reducing the equation $\text{mod }{\pi}$ might result in some multiple factors in $\overline{h(t_{k})}$, or even worse: the equation might be reducible.
\begin{itemize}
\item If Equation \ref{NormEq1} is reducible, then the residue field extension is an isomorphism and we have $\overline{y}'=\tilde{h}$ for some ${\tilde{h}}$.
\item If Equation \ref{NormEq1} is irreducible and has multiple factors, we normalize to obtain a new equation. The residue field extension then has degree $2$.
\end{itemize} 
%This explicitly describes the residue field of the component $\Gamma$.

\begin{rem}
The unique components of $R_{i}$ and $R_{i-1}$ have a single intersection point $P$ on the semistable model $\mathcal{D}_{S}$. Note that this intersection point is not visible in these equations: it is given on the special fiber as "$t_{i}=\infty$". To illustrate this, consider the algebra $R[x,t_{1}]/(t_{1}\pi-(x-a_{0}))$. The natural missing algebra is then given by $R[x,t'_{1}]/((x-a_{0})t'_{1}-\pi)$. 
On the overlap of these affine charts, we find that $t_{1}$ and $t'_{1}$ satisfy $t_{1}\cdot{t'_{1}}=1$. 
%They are glued together on $D(t_{1})$ and $D(t'_{1})$, with corresponding algebra $R[x,t'_{1},t_{1}]$ (omitting the relations), where $t'_{1}$ and $t_{1}$ satisfy the relation $t_{1}\cdot{t'_{1}}=1$.
It is clear from this equation why the intersection point with $t_{1}=0$ is not visible in the other chart.  
\end{rem}

\end{appendix}

\bibliographystyle{alpha}
\bibliography{bibfiles}{}

\end{document}